\documentclass{article}
\usepackage{amsmath}
\usepackage{amsfonts}
\usepackage{latexsym}
\usepackage{amsthm}
\usepackage{mathrsfs}
\usepackage{amscd}
\usepackage{amsmath}
\usepackage{amssymb}
\usepackage[all]{xy}

\newcommand{\Hom}{\mathop{\mathrm{Hom}}\nolimits}
\newcommand{\Ker}{\mathop{\mathrm{Ker}}\nolimits}
\newcommand{\Image}{\mathop{\mathrm{Im}}\nolimits}
\newcommand{\Spec}{\mathop{\mathrm{Spec}}\nolimits}

\newcommand{\colim}[1]{\mathop{\underset{#1}{\mathrm{colim}}}\nolimits}
\newcommand{\Mod}{\mathcal{M}od_k}
\newcommand{\Smk}{\mathcal{S}m_k}
\newcommand{\Shvk}{\mathcal{S}hv_k}

\newcommand{\aone}{\mathbb{A}^1}
\newcommand{\HAone}[1]{\mathop{\mathbf{H}_{#1}^{\mathbb{A}^1}}\nolimits}

\newtheorem{thm}{Theorem}[section]
\newtheorem{thmi}{Theorem}
\newtheorem{lem}[thm]{Lemma}
\newtheorem{prop}[thm]{Proposition}
\newtheorem{cor}[thm]{Corollary}
\theoremstyle{definition}
\newtheorem{defi}[thm]{Definition}
\newtheorem{examp}[thm]{Example}
\newtheorem{rem}[thm]{Remark}

\newtheorem*{nota}{Conventions}
\newtheorem*{acknow}{Acknowledgments}

\makeatletter

\@addtoreset{equation}{section}
\makeatother

\begin{document}

\title{Universal birational invariants and $\aone$-homology\footnote{Mathematics Subject Classification: Primary 14F42; Secondary 14E99, 14G05, 19E15, 14F05}}
\author{Yuri Shimizu\footnote{Department of Mathematics, Tokyo Institute of Technology, Tokyo, Japan \newline E-mail: qingshuiyouli27@gmail.com}}
\maketitle

\begin{abstract}
Let $k$ be a field admitting a resolution of singularities. In this paper, we prove that the functor of zeroth $\aone$-homology $\HAone{0}$ is universal as a functorial birational invariant of smooth proper $k$-varieties taking values in a category enriched by abelian groups. For a smooth proper $k$-variety $X$, we also prove that the dimension of $\HAone{0}(X;\mathbb{Q})(\Spec k)$ coincides with the number of $R$-equivalence classes of $X(k)$. We deduce these results as consequences of the structure theorem that for a smooth proper $k$-variety $X$, the sheaf $\HAone{0}(X)$ is the free abelian presheaf generated by the birational $\aone$-connected components $\pi_0^{b\aone}(X)$ of Asok-Morel.
\end{abstract}

\setcounter{tocdepth}{1}
\tableofcontents

\section*{Introduction}

Let $k$ be a field. In \cite{Mo1}, Morel introduced the $\aone$-homology theory of smooth (separated and of finite type) $k$-schemes. In the unstable $\aone$-homotopy theory introduced by Morel-Voevodsky \cite{MV}, the $\aone$-homology plays a role of the ordinary homology of topological spaces. For each $n \geq 0$, the $n$-th $\aone$-homology of a smooth $k$-scheme $X$ is denoted by $\HAone{n}(X)$. This is a Nisnevich sheaf of abelian groups on the category of smooth $k$-schemes $\Smk$. As an application of $\aone$-homology, Asok \cite{As} proved that the zeroth $\aone$-homology is a birational invariant of smooth proper $k$-varieties if $k$ is infinite. In this paper, we study the structure of the zeroth $\aone$-homology of smooth proper $k$-varieties and give some applications. In particular, we prove that the zeroth $\aone$-homology functor $\HAone{0}$ is universal as a functorial birational invariant of smooth proper $k$-varieties. 

First, we state the universal birational invariance of the zeroth $\aone$-homology. Let $\Smk^{prop}$ be the category of smooth proper $k$-schemes and $\mathbf{Im}\HAone{0}$ be the full subcategory of the category of abelian presheaves on $\Smk$ spanned by objects isomorphic to $\HAone{0}(X)$ for some $X \in \Smk^{prop}$. Our first theorem is stated as follows.

\begin{thmi}[see Theorem \ref{univ. birat. of H0A1}]\label{Intro. univ. br.}
Assume $k$ admits a resolution of singularities. Let $\mathcal{A}$ be an arbitrary category enriched by abelian groups (\textit{e.g.} an abelian category).
\begin{enumerate}
\item Let $F : \Smk^{prop} \to \mathcal{A}$ be an arbitrary functor which sends each birational morphism to an isomorphism. Then there exists one and only one (up to a natural equivalence) additive functor ${F_{S_b} : \mathbf{Im}\HAone{0} \to \mathcal{A}}$ such that the diagram
\begin{equation*}
\xymatrix{
\Smk^{prop} \ar[r]^-F \ar[d]_-{\HAone{0}} &\mathcal{A} \\
\mathbf{Im}\HAone{0} \ar@{.>}[ru]_-{F_{S_b}}
}
\end{equation*}
is $2$-commutative, \textit{i.e.}, we have a natural equivalence $F \cong F_{S_b} \circ \HAone{0}$.
\item Let $F' : (\Smk^{prop})^{op} \to \mathcal{A}$ be an arbitrary functor which sends each birational morphism to an isomorphism. Then there exists one and only one (up to a natural equivalence) additive functor ${F'_{S_b} : (\mathbf{Im}\HAone{0})^{op} \to \mathcal{A}}$ such that the diagram
\begin{equation*}
\xymatrix{
(\Smk^{prop})^{op} \ar[r]^-{F'} \ar[d]_-{(\HAone{0})^{op}} &\mathcal{A} \\
(\mathbf{Im}\HAone{0})^{op} \ar@{.>}[ru]_-{F'_{S_b}}
}
\end{equation*}
is $2$-commutative.
\end{enumerate}
\end{thmi}

In this theorem, we may replace $\Smk^{prop}$ with other full subcategory of $\Smk^{prop}$ like the category of projective varieties (see Definition \ref{def. br. ff} and Example \ref{ex. br. ff}). A similar result also holds in a motivic situation (see Theorem \ref{univ. br. inv. tr.}). In this case, $\Smk^{prop}$ is replaced with the full subcategory of the category of finite correspondences consisting of proper schemes and the $\aone$-homology with the zeroth homology of Voevodsky's motives, called Suslin homology sheaves (cf. \cite{Vo00}). The birational invariance theorem of Asok \cite{As} is assumed $k$ infinite, but this assumption is not necessary for our proof. In general, we prove that a proper birational morphism of (not necessary proper) varieties over an arbitrary field induces an isomorphism of zeroth $\aone$-homology sheaves. (see Proposition \ref{prop. br. inv.}).

Second, we relate the zeroth $\aone$-homology to rational points. We consider the zeroth $\aone$-homology with $\mathbb{Q}$-coefficients $\HAone{0}(X;\mathbb{Q})$ for a smooth proper $k$-variety $X$.

\begin{thmi}[see Theorem \ref{A1-betti and rat pt}]\label{Intro. rat. pt.}
Assume $k$ admits a resolution of singularities. For a smooth proper $k$-variety $X$, we have
\begin{equation*}
\dim_{\mathbb{Q}} \HAone{0}(X;\mathbb{Q})(\Spec k) = \#(X(k)/R).
\end{equation*}
Moreover, $\HAone{0}(X)(\Spec k) = 0$ if and only if $X(k) = \emptyset$.
\end{thmi}

Here $X(k)/R$ is the quotient set of $X(k)$ by the $R$-equivalence introduced by Manin \cite{Ma}. Note that the abelian group $\HAone{0}(X)(\Spec k)$ can also be expressed in terms of triangulated categories (see Remark \ref{other expressions of b0A1}).

Third, we state a structure theorem of zeroth $\aone$-homology sheaves. In \cite{AM}, Asok-Morel constructed a Nisnevich sheaf of sets $\pi_0^{b\aone}(X)$ on $\Smk$, called the birational $\aone$-connected components, for each $X \in \Smk^{prop}$. We generalize this construction for all smooth $k$-schemes. Our structure theorem is stated as follows.

\begin{thmi}[see Theorem \ref{str H0A1}]\label{Intro. str. A1}
Assume $k$ admits a resolution of singularities. For every $X \in \Smk$, there exists a natural epimorphism of presheaves
\begin{equation*}
\mathbb{Z}_{pre}(\pi_0^{b\aone}(X)) \twoheadrightarrow \HAone{0}(X).
\end{equation*}
Moreover, this is an isomorphism if $X$ is proper.
\end{thmi}

Here $\mathbb{Z}_{pre}(\pi_0^{b\aone}(X))$ is the free abelian presheaf generated by $\pi_0^{b\aone}(X)$. A similar result also holds for Suslin homology sheaves (see Theorem \ref{str thm of Suslin}). Theorem \ref{Intro. str. A1} has some applications to $\aone$-homotopy theory (see \S\ref{App}). Moreover, the Suslin homology version of Theorem \ref{Intro. str. A1} also has an application to zero cycles (see Corollary \ref{cor Sb-1Cor isom CH0}). Theorems \ref{Intro. univ. br.} and \ref{Intro. rat. pt.} are consequences of Theorem \ref{Intro. str. A1}.

This paper is organized as follows. In \S\ref{BSALOC}, we recall and prepare basic results on birational sheaves and the localization of some categories of smooth $k$-schemes by birational morphisms. In \S\ref{BAC}, we generalize birational $\aone$-connected components of Asok-Morel for all Nisnevich sheaves on $\Smk$. In \S\ref{BSOM}, we give a relationship between birational sheaves and strictly $\aone$-invariant sheaves. In \S\ref{ST}, we prove Theorems \ref{Intro. rat. pt.} and \ref{Intro. str. A1}. In \S\ref{UBI}, we prove Theorem \ref{Intro. univ. br.}. In \S\ref{App}, we give applications to $\aone$-homotopy theory.

\begin{nota}
Throughout this paper, we fix a field $k$ and a commutative unital ring $\Lambda$. All $k$-varieties are assumed irreducible but not assumed geometrically irreducible. The field $k$ is said to admit a resolution of singularities, if $k$-varieties always have a resolution of singularities. Let $\Smk$ be the category of separated and smooth $k$-schemes of finite type. Objects of $\Smk$ are simply called smooth $k$-schemes. We regard every smooth $k$-scheme $X$ as the presheaf of sets on $\Smk$ represented by $X$. For a presheaf $\mathscr{F}$ on $\Smk$ and an affine scheme $\Spec A$ which is the limit of an inversed system $\{U_\lambda \}_\lambda$ in $\Smk$, we write 
\begin{equation*}
\mathscr{F}(A) = \colim{\lambda} \mathscr{F}(U_\lambda).
\end{equation*}
For a category $\mathcal{C}$, we denote $\mathcal{P}resh(\mathcal{C})$ (resp. $\mathcal{P}resh(\mathcal{C},\Lambda)$) for the category of presheaves of sets (resp. $\Lambda$-modules) on $\mathcal{C}$. A diagram of categories
\begin{equation*}
\begin{CD}
\mathcal{C} @>F>> \mathcal{D} \\
@V{G'}VV @VV{G}V \\
\mathcal{C'} @>{F'}>> \mathcal{D'}
\end{CD}
\end{equation*}
is called $2$-commutative, if there exists a natural equivalence $G \circ F \cong F' \circ G'$. For a local ring $A$, we denote $\kappa_A$ for its residue field.
\end{nota}

\begin{acknow}
I would like to thank my adviser Shohei Ma for many useful advices. I also would like to thank Tom Bachmann for a helpful comment. This work was supported by JSPS KAKENHI Grant Number JP19J21433.
\end{acknow}

\section{Birational sheaves and localizations of categories}\label{BSALOC}

In this section, we consider localizations of categories in the sense of Gabriel-Zisman \cite{GZ}. Let $\mathcal{C}$ be a category and $S$ be a family of morphisms in $\mathcal{C}$. Recall that the localization of $\mathcal{C}$ by $S$ is a category $S^{-1}\mathcal{C}$ with a morphism $\mathscr{L} : \mathcal{C} \to S^{-1}\mathcal{C}$ such that
\begin{itemize}
\item for every $s \in S$, the image $\mathscr{L}(s)$ is an isomorphism in $S^{-1}\mathcal{C}$, and
\item for every category $\mathcal{D}$ and every functor $\mathcal{C} \to \mathcal{D}$ which sends each $s \in S$ to an isomorphism in $\mathcal{D}$, there exists one and only one (up to a natural equivalence) functor $S^{-1}\mathcal{C} \to \mathcal{D}$ such that the diagram
\begin{equation*}
\xymatrix{
\mathcal{C} \ar[r] \ar[d] &\mathcal{D} \\
S^{-1}\mathcal{C} \ar@{.>}[ru]
}
\end{equation*}
is $2$-commutative.
\end{itemize}
Note that the induced functor $\mathcal{P}resh(S^{-1}\mathcal{C}) \to \mathcal{P}resh(\mathcal{C})$ is fully faithful and its essential image is spanned by presheves which send each $s \in S$ for a bijection. We especially treat some cases where $\mathcal{C}$ is a category of smooth $k$-schemes and $S$ consists of all birational morphisms in $\mathcal{C}$.

\subsection{Localizations by birational morphisms}

Let $\Smk^{var}$ (resp. $\Smk^{prop}$, $\Smk^{pv}$) be the full subcategory of $\Smk$ consisting of irreducible (resp. proper, proper and irreducible) $k$-schemes. For a subcategory $\mathcal{C} \subseteq \Smk$, we write $S_b^{-1}\mathcal{C}$ for the localization of $\mathcal{C}$ by birational morphisms within $\mathcal{C}$. In \cite{KS1}, Kahn-Sujatha proved that the inclusion $\Smk^{pv} \hookrightarrow \Smk^{var}$ induces an equivalence of categories
\begin{equation}\label{eq. cat Smpv Smvar}
S_b^{-1}\Smk^{pv} \xrightarrow{\cong} S_b^{-1}\Smk^{var}
\end{equation}
if $k$ admits a resolution of singularities (see \cite[Prop. 8.5]{KS1}).

Let $\mathcal{C}or_{k,\Lambda}$ be the category of finite correspondences of Voevodsky with coefficients in a commutative unital ring $\Lambda$ (see definition \cite[Def. 1.5]{MVW}). We denote $\Gamma$ for the canonical functor $\Smk \to \mathcal{C}or_{k,\Lambda}$ and write
\begin{equation*}
{\mathcal{C}or_{k,\Lambda}(X,Y) = \Hom_{\mathcal{C}or_{k,\Lambda}}(X,Y)}
\end{equation*}
for each $X,Y \in \mathcal{C}or_{k,\Lambda}$. For a subcategory $\mathcal{C} \subseteq \mathcal{C}or_{k,\Lambda}$, we denote $S_b^{-1}\mathcal{C}$ for the localization of $\mathcal{C}$ by finite correspondences associated with a birational morphism. Let $\mathcal{C}or_{k,\Lambda}^{var}$ (resp. $\mathcal{C}or_{k,\Lambda}^{prop}$, $\mathcal{C}or_{k,\Lambda}^{pv}$) be the full subcategory of $\mathcal{C}or_{k,\Lambda}$ consisting of irreducible (resp. proper, proper and irreducible) $k$-schemes. Next, we give an analogue of the equivalence \eqref{eq. cat Smpv Smvar} for categories of finite correspondences. For this, we prove the following lemma.

\begin{lem}\label{inj cor birat}
For all $X_0, X, Y \in \mathcal{C}or_{k,\Lambda}$ and every birational morphism ${f : X_0 \to X}$ in $\Smk$, the induced map
\begin{equation*}
f^* : \mathcal{C}or_{k,\Lambda}(X,Y) \to \mathcal{C}or_{k,\Lambda}(X_0,Y); c \mapsto c \circ \Gamma(f)
\end{equation*}
is injective.
\end{lem}

\begin{proof}
Let $i : U \hookrightarrow X$ be a dense open embedding that $f^{-1}(i(U)) \to i(U)$ is an isomorphism. By the commutativity of the diagram
\begin{equation*}
\begin{CD}
\mathcal{C}or_{k,\Lambda}(X,Y) @>{f^*}>> \mathcal{C}or_{k,\Lambda}(X_0,Y) \\
@V{i^*}VV @VVV \\
\mathcal{C}or_{k,\Lambda}(U,Y) @>{\cong}>> \mathcal{C}or_{k,\Lambda}(f^{-1}(U),Y),
\end{CD}
\end{equation*}
we only need to show that $i^*$ is injective. Note that for every $c \in \mathcal{C}or_{k,\Lambda}(X,Y)$ the finite correspondence $i^*c \in \mathcal{C}or_{k,\Lambda}(U,Y)$ coincides with the pullback of $c$ by the open embedding
\begin{equation*}
j : U \times Y \hookrightarrow X \times Y
\end{equation*}
as an algebraic cycle of $U \times Y$. We write $W = (X \times Y) - j(U \times Y)$. Now we obtain an equality
\begin{equation*}
\Ker(Z(X \times Y) \xrightarrow{j^*} Z(U \times Y)) = Z(W),
\end{equation*}
where $Z(-)$ means the set of algebraic cycles with $\Lambda$-coefficients. Thus we have
\begin{align*}
&\Ker(\mathcal{C}or_{k,\Lambda}(X,Y) \xrightarrow{i^*} \mathcal{C}or_{k,\Lambda}(U,Y)) \\
&= \mathcal{C}or_{k,\Lambda}(X,Y) \cap \Ker(Z(X \times Y) \xrightarrow{j^*} Z(U \times Y)) \\
&= \mathcal{C}or_{k,\Lambda}(X,Y) \cap Z(W).
\end{align*}
On the other hand, since $U \times Y$ is non-empty, the map $W \to X$ is not surjective. Therefore, all non-zero elements of $Z(W)$ are not finite correspondences $X \to Y$. Thus we have
\begin{equation*}
\Ker i^* = \mathcal{C}or_{k,\Lambda}(X,Y) \cap Z(W) = 0.
\end{equation*}
\end{proof}

We prove that $S_b^{-1}\mathcal{C}or_{k,\Lambda}^{pv} \cong S_b^{-1}\mathcal{C}or^{var}_{k,\Lambda}$.

\begin{prop}\label{corpv eq corvar}
Assume $k$ admits a resolution of singularities. Then the inclusion $\mathcal{C}or_{k,\Lambda}^{pv} \hookrightarrow \mathcal{C}or^{var}_{k,\Lambda}$ induces an equivalence of categories
\begin{equation*}
S_b^{-1}\mathcal{C}or_{k,\Lambda}^{pv} \xrightarrow{\cong} S_b^{-1}\mathcal{C}or^{var}_{k,\Lambda}.
\end{equation*}
\end{prop}

\begin{proof}
By \cite[Thm. 2.1 and 4.3]{KS1}, we only need to show that the functors $\mathcal{C}or_{k,\Lambda}^{pv} \hookrightarrow \mathcal{C}or^{var}_{k,\Lambda}$ and $\mathcal{C}or_{k,\Lambda}^{prop} \hookrightarrow \mathcal{C}or_{k,\Lambda}$ satisfy the conditions (b1)-(b3) in \cite[Prop. 5.10]{KS1}. Then (b1) follows from Lemma \ref{inj cor birat} and (b2) and (b3) follow from a similar argument as \cite[proof of Prop. 8.5]{KS1} (see also \cite[proof of Prop. 8.4]{KS1}).
\end{proof}

\begin{rem}
By a similar proof, we also obtain an equivalence of categories
\begin{equation*}
S_b^{-1}\mathcal{C}or_{k,\Lambda} \xrightarrow{\cong} S_b^{-1}\mathcal{C}or^{prop}_{k,\Lambda}.
\end{equation*}
\end{rem}

\subsection{Birational sheaves}

Following Asok-Morel \cite{AM}, we call a presheaf $\mathscr{F}$ on $\Smk$ a \textit{birational sheaf}, if
\begin{itemize}
\item[\textbf{B1}] the canonical map $\mathscr{F}(U \sqcup V) \to \mathscr{F}(U) \times \mathscr{F}(V)$ is bijective for all ${U,V \in \Smk}$, and
\item[\textbf{B2}] every open embedding $U \hookrightarrow X$ in $\Smk$ induces a bijection ${\mathscr{F}(X) \xrightarrow{\cong} \mathscr{F}(U)}$.  
\end{itemize}
Note that birational sheaves are always Nisnevich (see \cite[Lem. 6.1.2]{AM}). Let $\Shvk$ be the category of Nisnevich sheaves of sets on $\Smk$. We denote $\Shvk^{br}$ for the full subcategory of $\Shvk$ consisting of birational sheaves. We give a canonical equivalence $\Shvk^{br} \cong \mathcal{P}resh(S_b^{-1}\Smk^{var})$.

\begin{lem}\label{eq. br. sh.}
There exists an equivalence of categories
\begin{equation*}
\Shvk^{br} \xrightarrow{\cong} \mathcal{P}resh(S_b^{-1}\Smk^{var})
\end{equation*}
such that the diagram
\begin{equation*}
\begin{CD}
\Shvk^{br} @>{\cong}>> \mathcal{P}resh(S_b^{-1}\Smk^{var}) \\
@VVV @VVV \\
\Shvk @>>> \mathcal{P}resh(\Smk^{var})
\end{CD}
\end{equation*}
is $2$-commutative.
\end{lem}

\begin{proof}
Let $\mathcal{P}resh_{\mathbf{B2}}(\Smk^{var})$ be the full subcategory of $\mathcal{P}resh(\Smk^{var})$ consisting of presheaves which satisfy \textbf{B2}. Then $\mathcal{P}resh_{\mathbf{B2}}(\Smk^{var})$ is equivalent to $\mathcal{P}resh(S_b^{-1}\Smk^{var})$ by the universality of localizations. On the other hand, since every object in $\Smk$ is a finite coproduct of objects in $\Smk^{var}$, the restriction functor $\Shvk^{br} \to \mathcal{P}resh_{\mathbf{B2}}(\Smk^{var})$ is an equivalence by \textbf{B1}. Thus we have a $2$-commutative the diagram
\begin{equation*}
\begin{CD}
\Shvk^{br} @>{\cong}>> \mathcal{P}resh_{\mathbf{B2}}(\Smk^{var}) @<{\cong}<< \mathcal{P}resh(S_b^{-1}\Smk^{var}) \\
@VVV @VVV @VVV\\
\Shvk @>>> \mathcal{P}resh(\Smk^{var}) @= \mathcal{P}resh(\Smk^{var}).
\end{CD}
\end{equation*}
\end{proof}

For a category $\mathcal{C}$ enriched by $\Lambda$-modules, we denote $\mathcal{L}in(\mathcal{C},\Lambda)$ for the category of $\Lambda$-linear presheaves of $\Lambda$-modules on $\mathcal{C}$. We write
\begin{equation*}
\mathbf{PST}_k(\Lambda) = \mathcal{L}in(\mathcal{C}or_{k,\Lambda},\Lambda).
\end{equation*}
Objects in $\mathbf{PST}_k(\Lambda)$ are called a presheaves with transfers. A presheaf with transfers $M$ is called Nisnevich (resp. birational), if so is the direct image $\Gamma_*M$ by the canonical functor $\Gamma : \Smk \to \mathcal{C}or_{k,\Lambda}$. We denote $\mathbf{NST}_k(\Lambda)$ (resp. $\mathbf{NST}_k^{br}(\Lambda)$) for the category of Nisnevich (resp. birational) sheaves with transfers. We prove an analogue of Lemma \ref{eq. br. sh.} for sheaves with transfers.

\begin{lem}\label{eq. br. sh. cor}
There exists an equivalence of categories
\begin{equation*}
\mathbf{NST}_k^{br}(\Lambda) \xrightarrow{\cong} \mathcal{L}in(S_b^{-1}\mathcal{C}or^{var}_{k,\Lambda},\Lambda)
\end{equation*}
such that the diagram
\begin{equation*}
\begin{CD}
\mathbf{NST}_k^{br}(\Lambda) @>{\cong}>> \mathcal{L}in(S_b^{-1}\mathcal{C}or^{var}_{k,\Lambda},\Lambda) \\
@VVV @VVV \\
\mathbf{NST}_k(\Lambda) @>>> \mathcal{L}in(\mathcal{C}or^{var}_{k,\Lambda},\Lambda)
\end{CD}
\end{equation*}
is $2$-commutative.
\end{lem}

\begin{proof}
Let $\mathcal{L}in_{\mathbf{B2}}(\mathcal{C}or^{var}_{k,\Lambda},\Lambda)$ be the full subcategory of $\mathcal{L}in(\mathcal{C}or^{var}_{k,\Lambda},\Lambda)$ consisting of presheaves which satisfy \textbf{B2}. By a similar argument as the proof of Lemma \ref{eq. br. sh.}, we also have a $2$-commutative diagram
\begin{equation*}
\begin{CD}
\mathbf{NST}_k^{br}(\Lambda) @>{\cong}>> \mathcal{L}in_{\mathbf{B2}}(\mathcal{C}or^{var}_{k,\Lambda},\Lambda) @<{\cong}<< \mathcal{L}in(S_b^{-1}\mathcal{C}or^{var}_{k,\Lambda},\Lambda) \\
@VVV @VVV @VVV\\
\mathbf{NST}_k(\Lambda) @>>> \mathcal{L}in(\mathcal{C}or^{var}_{k,\Lambda},\Lambda) @= \mathcal{L}in(\mathcal{C}or^{var}_{k,\Lambda},\Lambda).
\end{CD}
\end{equation*}
\end{proof}

The following lemma says that birational sheaves are $\aone$-invariant.

\begin{lem}\label{A1-inv of br sh}
For every $\mathscr{F} \in \Shvk^{br}$ and every $U \in \Smk^{var}$, the canonical map
\begin{equation*}
\mathscr{F}(U) \to \mathscr{F}(U \times \aone) 
\end{equation*}
is an isomorphism.
\end{lem}

\begin{proof}
By \cite[Thm. 1.7.9]{KS2}, stable birational morphisms are isomorphisms in $S_b^{-1}\Smk^{var}$. Since the projection $U \times \aone \to U$ is stable biratioal, we have $\mathscr{F}(U) \cong \mathscr{F}(U \times \aone)$ by Lemma \ref{eq. br. sh.}.
\end{proof}

\section{Birational $\aone$-connected components}\label{BAC}

In \cite{AM}, Asok-Morel introduced a birational sheaf $\pi_0^{b\aone}(X)$, called the \textit{birational $\aone$-connected components}, for each $X \in \Smk^{prop}$. In this section, we generalize this construction for all Nisnevich sheaves on $\Smk$. We first construct a functor $\pi_0^{br} : \Shvk \to \Shvk^{br}$ and secondly prove that $\pi_0^{br}(X) \cong \pi_0^{b\aone}(X)$ for every $X \in \Smk^{prop}$. Moreover, we also prove that the functor $\pi_0^{br}$ is left adjoint to the inclusion $\Shvk^{br} \hookrightarrow \Shvk$. Similarly, we construct a left adjoint functor of the inclusion $\mathbf{NST}_k^{br}(\Lambda) \hookrightarrow \mathbf{NST}_k(\Lambda)$.

\subsection{Birational $\aone$-connected components of sheaves}

First, we give a functor $\pi_0^{br} : \Shvk \to \Shvk^{br}$.

\begin{defi}
We define a functor $\pi_0^{br} : \Shvk \to \Shvk^{br}$ as the composition
\begin{equation*}
\Shvk \hookrightarrow \mathcal{P}resh(\Smk) \to \mathcal{P}resh(\Smk^{var}) \to \Shvk^{br},
\end{equation*}
where the third functor is a left Kan extension of
\begin{equation*}
\Smk^{var} \to S_b^{-1}\Smk^{var} \hookrightarrow \mathcal{P}resh(S_b^{-1}\Smk^{var}) \cong \Shvk^{br}.
\end{equation*}
\end{defi}

Explicitly, every sheaf $\mathscr{F}$ on $\Smk$ is canonically isomorphic to the colimit of a diagram $\{X_{\lambda}\}_{\lambda}$ in $\Smk^{var}$ (note that every object in $\Smk$ is a coproduct of smooth $k$-varieties). Then by the definition, we have a natural bijection
\begin{equation*}\label{explicit pi0br.}
\pi_0^{br}(\mathscr{F})(U) \cong \colim{\lambda} \Hom_{S_b^{-1}\Smk^{var}}(U,X_{\lambda})
\end{equation*}
for every $U \in \Smk$. Thus the map
\begin{equation*}
\Hom_{\Smk}(U,X_{\lambda}) \to \Hom_{S_b^{-1}\Smk^{var}}(U,X_{\lambda})
\end{equation*}
induces a natural morphism
\begin{equation*}
\mathscr{F} \to \pi_0^{br}(\mathscr{F})
\end{equation*}
in $\Shvk$ which is functorial for $\mathscr{F}$.

Let $X$ be a smooth proper $k$-variety. Recall that two points in $X(K)$ for a field extension $K/k$ are called $R$-equivalence (in the sense of Manin \cite{Ma}), if these points are connected by the image of a chain of $k$-morphisms $\mathbb{P}^1_K \to X$. The quotient set of $X(K)$ by the $R$-equivalence is denoted by $X(K)/R$. The birational $\aone$-connected components $\pi_0^{b\aone}(X)$ has the property that there exists a canonical bijection
\begin{equation*}
X(K)/R \cong \pi_0^{b\aone}(X)(K)
\end{equation*}
for every finitely generated separable extension $K/k$ (see \cite[Thm. 6.2.1]{AM}). Our sheaf $\pi_0^{br}(X)$ also has the same property.

\begin{lem}\label{Req and pibr}
Let $X$ be a smooth proper $k$-scheme and $K/k$ be a finitely generated separable extension. Then the map $X(K) \to \pi_0^{br}(X)(K)$ induced by the canonical morphism $X \to \pi_0^{br}(X)$ factors though a bijection
\begin{equation*}
X(K)/R \cong \pi_0^{br}(X)(K).
\end{equation*}
\end{lem}

\begin{proof}
This follows from the bijection
\begin{equation*}
X(K)/R \cong \Hom_{S_b^{-1}\Smk^{var}}(U,X)
\end{equation*}
in \cite[Thm. 6.6.3]{KS2}, where $U$ is a smooth model of $K$.
\end{proof}

Next, we give an isomorphism $\pi_0^{br}(X) \cong \pi_0^{b\aone}(X)$ for each $X \in \Smk^{prop}$.

\begin{prop}\label{br isom bA1}
For every $X \in \Smk^{prop}$, there exists an isomorphism of birational sheaves
\begin{equation*}
\pi_0^{br}(X) \cong \pi_0^{b\aone}(X).
\end{equation*}
\end{prop}

Before proving this proposition, we recall the category $\mathcal{F}_k^r-\mathcal{S}et$ introduced by Asok-Morel \cite{AM}. Let $\mathcal{F}_k$ be the category of finitely generated separable extension fields over $k$. Each object $\mathcal{S} \in \mathcal{F}_k^r-\mathcal{S}et$ is a covariant functor
\begin{equation*}
\mathcal{F}_k \to \mathcal{S}et
\end{equation*}
together with a map $\mathcal{S}(K) \to \mathcal{S}(\kappa_A)$ for each $K \in \mathcal{F}_k$ and its discrete valuation ring $A$ with $\kappa_A \in \mathcal{F}_k$. Moreover, a morphism $\mathcal{S} \to \mathcal{S'}$ in $\mathcal{F}_k^r-\mathcal{S}et$ is a natural transformation $\mathcal{S} \to \mathcal{S'}$ such that the diagram
\begin{equation*}
\begin{CD}
\mathcal{S}(K) @>>> \mathcal{S}(\kappa_A) \\
@VVV @VVV \\
\mathcal{S'}(K) @>>> \mathcal{S'}(\kappa_A)
\end{CD}
\end{equation*}
commutes. The restriction of each $\mathscr{F} \in \Shvk^{br}$ on $\mathcal{F}_k$ together with the map
\begin{equation*}
\mathscr{F}(K) \cong \mathscr{F}(A) \to \mathscr{F}(\kappa_A)
\end{equation*}
is an object of $\mathcal{F}_k^r-\mathcal{S}et$. Thus we have the restriction functor
\begin{equation}\label{restriction functor}
\mathrm{Res} : \Shvk^{br} \to \mathcal{F}_k^r-\mathcal{S}et.
\end{equation}

\begin{proof}[Proof of Proposition \ref{br isom bA1}]
By Lemma \ref{A1-inv of br sh} and \cite[Thm. 6.1.7]{AM}, the restriction functor \eqref{restriction functor} is fully faithful. Thus we only need to construct an isomorphism
\begin{equation}\label{Res isom br bA1}
\mathrm{Res}(\pi_0^{br}(X)) \cong \mathrm{Res}(\pi_0^{b\aone}(X))
\end{equation}
in $\mathcal{F}_k^r-\mathcal{S}et$. Lemma \ref{Req and pibr} and \cite[Thm. 6.2.1]{AM} give a commutative diagram
\begin{equation*}
\begin{CD}
X(K) @= X(K) @= X(K) \\
@VVV @VVV @VVV \\
\pi_0^{br}(X)(K) @>{\cong}>> X(K)/R @<{\cong}<< \pi_0^{b\aone}(X)(K)
\end{CD}
\end{equation*}
for all $K \in \mathcal{F}_k$. Then the vertical maps are surjective. Thus for every extension field ${L/K}$ which is finitely generated and separable over $k$, the obviously commutative diagram
\begin{equation*}
\begin{CD}
X(K) @= X(K) @= X(K) \\
@VVV @VVV @VVV \\
X(L) @= X(L) @= X(L)
\end{CD}
\end{equation*}
shows that
\begin{equation*}
\begin{CD}
\pi_0^{br}(X)(K) @>{\cong}>> X(K)/R @<{\cong}<< \pi_0^{b\aone}(X)(K) \\
@VVV @VVV @VVV \\
\pi_0^{br}(X)(L) @>{\cong}>> X(L)/R @<{\cong}<< \pi_0^{b\aone}(X)(L)
\end{CD}
\end{equation*}
commutes. Similarly, for every discrete valuation ring $A$ of $K$ with $\kappa_A \in \mathcal{F}_k$, the diagram
\begin{equation*}
\begin{CD}
X(K) @= X(K) @= X(K) \\
@VVV @VVV @VVV \\
X(\kappa_A) @= X(\kappa_A) @= X(\kappa_A)
\end{CD}
\end{equation*}
shows that
\begin{equation*}
\begin{CD}
\pi_0^{br}(X)(K) @>{\cong}>> X(K)/R @<{\cong}<< \pi_0^{b\aone}(X)(K) \\
@VVV @VVV @VVV \\
\pi_0^{br}(X)(\kappa_A) @>{\cong}>> X(\kappa_A)/R @<{\cong}<< \pi_0^{b\aone}(X)(\kappa_A)
\end{CD}
\end{equation*}
also commutes. Here the map $X(K) \to X(\kappa_A)$ is the composition of the inverse of the bijection $X(K) \xleftarrow{\cong} X(A)$ (the bijectivity follows from the valuative criterion of properness) and the canonical map $X(A) \to X(\kappa_A)$. Therefore, the bijection
\begin{equation*}
\pi_0^{br}(X)(K) \cong \pi_0^{b\aone}(X)(K)
\end{equation*}
gives an isomorphism \eqref{Res isom br bA1}.
\end{proof}

From now on, we write $\pi_0^{b\aone} = \pi_0^{br}$. We have an adjunction $\Shvk \rightleftarrows \Shvk^{br}$.

\begin{lem}\label{adj pi b aone}
The functor $\pi_0^{b\aone}$ is left adjoint to the inclusion $\Shvk^{br} \hookrightarrow \Shvk$.
\end{lem}

\begin{proof}
Let $\mathscr{G}$ be an arbitrary birational sheaf. For every $X \in \Smk^{var}$, Yoneda's lemma in $\Smk$ gives a natural isomorphism
\begin{equation*}
\Hom_{\Shvk}(X,\mathscr{G}) \cong \mathscr{G}(X).
\end{equation*}
Under the identification by the equivalence in Lemma \ref{eq. br. sh.}, the presheaf $\pi_0^{b\aone}(X)$ is represented by $X$ in $S_b^{-1}\Smk^{var}$. Thus Yoneda's lemma in $S_b^{-1}\Smk^{var}$ also gives an isomorphism
\begin{equation*}
\Hom_{\Shvk^{br}}(\pi_0^{b\aone}(X),\mathscr{G}) \cong \mathscr{G}(X)
\end{equation*}
and we have
\begin{equation}\label{isom. adj. of br. for X}
\Hom_{\Shvk^{br}}(\pi_0^{b\aone}(X),\mathscr{G}) \cong \Hom_{\Shvk}(X,\mathscr{G}).
\end{equation}
Let $\mathscr{F}$ be a Nisnevich sheaf on $\Smk$ which is the colomit of a diagram $\{X_{\lambda}\}_{\lambda}$ in $\Smk^{var}$. Then \eqref{isom. adj. of br. for X} gives isomorphisms
\begin{align*}
\Hom_{\Shvk^{br}}(\pi_0^{b\aone}(\mathscr{F}),\mathscr{G}) &\cong \Hom_{\Shvk^{br}}(\colim{\lambda}\pi_0^{b\aone}(X_{\lambda}),\mathscr{G}) \\
&\cong \lim_{\lambda}\Hom_{\Shvk^{b\aone}}(\pi_0^{b\aone}(X_{\lambda}),\mathscr{G}) \\
&\cong \lim_{\lambda}\Hom_{\Shvk}(X_{\lambda},\mathscr{G}) \\
&\cong \Hom_{\Shvk}(\colim{\lambda}X_{\lambda},\mathscr{G}) \\
&\cong \Hom_{\Shvk}(\mathscr{F},\mathscr{G}).
\end{align*}
\end{proof}

\subsection{Birational $\aone$-connected components with transfers}

Next, we define a functor $\Lambda\pi_{0,tr}^{b\aone} : \mathbf{NST}_k(\Lambda) \to \mathbf{NST}_k^{br}(\Lambda)$. We denote $\Lambda_{tr}$ for the Yoneda embedding
\begin{equation*}
\mathcal{C}or_{k,\Lambda} \to \mathbf{NST}_k(\Lambda); X \mapsto \mathcal{C}or_{k,\Lambda}(-,X).
\end{equation*}

\begin{defi}
We define a functor $\Lambda\pi_{0,tr}^{b\aone} : \mathbf{NST}_k(\Lambda) \to \mathbf{NST}_k^{br}(\Lambda)$ as the composition
\begin{equation*}
\mathbf{NST}_k(\Lambda) \hookrightarrow \mathbf{PST}_k(\Lambda) \to \mathcal{L}in(\mathcal{C}or_{k,\Lambda}^{var},\Lambda) \to \mathbf{NST}_k^{br}(\Lambda),
\end{equation*}
where the third functor is a left Kan extension of
\begin{equation*}
\mathcal{C}or_{k,\Lambda}^{var} \to S_b^{-1}\mathcal{C}or_{k,\Lambda}^{var} \hookrightarrow \mathcal{L}in(S_b^{-1}\mathcal{C}or_{k,\Lambda}^{var},\Lambda) \cong \mathbf{NST}_k^{br}(\Lambda).
\end{equation*}
We simply write $\Lambda\pi_{0,tr}^{b\aone}(X) = \Lambda\pi_{0,tr}^{b\aone}(\Lambda_{tr}(X))$.
\end{defi}

We have an adjunction $\mathbf{NST}_k(\Lambda) \rightleftarrows \mathbf{NST}_k^{br}(\Lambda)$.

\begin{lem}\label{adj pi b aone tr}
The functor $\Lambda\pi_{0,tr}^{b\aone}$ is left adjoint to ${\mathbf{NST}_k^{br}(\Lambda) \hookrightarrow \mathbf{NST}_k(\Lambda)}$.
\end{lem}

\begin{proof}
Let $N$ be an arbitrary birational sheaf with transfers of $\Lambda$-modules. For every $X \in \mathcal{C}or_{k,\Lambda}^{var}$, Yoneda's lemma in $\mathcal{C}or_{k,\Lambda}$ gives an isomorphism
\begin{equation*}
\Hom_{\mathbf{NST}_k^{br}(\Lambda)}(\Lambda\pi_{0,tr}^{b\aone}(X),N) \cong N(X).
\end{equation*}
Under the identification by the equivalence in Lemma \ref{eq. br. sh. cor}, the presheaf $\Lambda\pi_{0,tr}^{b\aone}(X)$ is represented by $X$ in $S_b^{-1}\mathcal{C}or_{k,\Lambda}^{var}$. Thus Yoneda's lemma in $S_b^{-1}\mathcal{C}or_{k,\Lambda}^{var}$ shows that
\begin{equation}\label{eq of adj pi b aone tr}
\Hom_{\mathbf{NST}_k(\Lambda)}(\Lambda_{tr}(X),N) \cong N(X) \cong \Hom_{\mathbf{NST}_k^{br}(\Lambda)}(\Lambda\pi_{0,tr}^{b\aone}(X),N).
\end{equation}
Let $M$ be a Nisnevich sheaf with transfers such that
\begin{equation*}
M \cong \colim{\lambda}\Lambda_{tr}(X_{\lambda})
\end{equation*}
for a diagram $\{X_{\lambda}\}_{\lambda}$ in $\mathcal{C}or_{k,\Lambda}^{var}$. Then there exists a canonical isomorphism
\begin{equation*}
\Lambda\pi_{0,tr}^{b\aone}(M) \cong \colim{\lambda}\Lambda\pi_{0,tr}^{b\aone}(X_{\lambda})
\end{equation*}
by the definition of the functor $\Lambda\pi_{0,tr}^{b\aone}$. Thus \eqref{eq of adj pi b aone tr} gives isomorphisms
\begin{align*}
\Hom_{\mathbf{NST}_k^{br}(\Lambda)}(M,N) &\cong \Hom_{\mathbf{NST}_k^{br}(\Lambda)}(\colim{\lambda}\Lambda_{tr}(X_{\lambda}),N) \\
&\cong \lim_{\lambda} \Hom_{\mathbf{NST}_k^{br}(\Lambda)}(\Lambda_{tr}(X_{\lambda}),N) \\
&\cong \lim_{\lambda} \Hom_{\mathbf{NST}_k^{br}(\Lambda)}(\Lambda\pi_{0,tr}^{b\aone}(X_{\lambda}),N) \\
&\cong \Hom_{\mathbf{NST}_k^{br}(\Lambda)}(\colim{\lambda}\Lambda\pi_{0,tr}^{b\aone}(X_{\lambda}),N) \\
&\cong \Hom_{\mathbf{NST}_k^{br}(\Lambda)}(\Lambda\pi_{0,tr}^{b\aone}(M),N).
\end{align*}
\end{proof}

\begin{rem}
Lemma \ref{adj pi b aone tr} also gives a left adjoint of the forgetful functor $\mathbf{NST}_k^{br}(\Lambda) \to \Shvk$. Indeed, a left Kan extension $\Shvk \to \mathbf{NST}_k(\Lambda)$ of the composition $\Smk \xrightarrow{\Gamma} \mathcal{C}or_{k,\Lambda} \xrightarrow{\Lambda_{tr}} \mathbf{NST}_k(\Lambda)$ is left adjoint to the forgetful functor $\mathbf{NST}_k(\Lambda) \to \Shvk$.
\end{rem}

\section{Birational sheaves of modules}\label{BSOM}

Recall that a Nisnevich sheaf of $\Lambda$-modules $M$ on $\Smk$ is called \textit{strictly $\aone$-invariant}, if the map
\begin{equation*}
H^i_{Nis}(U,M) \to H^i_{Nis}(U \times \aone,M)
\end{equation*}
induced by the projection $U \times \aone \to U$ is an isomorphism for all $i \geq 0$ and all $U \in \Smk$. Our aim of this section is to construct a birational sheaf of $\Lambda$-modules $M^{br}$ with a monomorphism $\mu^M : M^{br} \hookrightarrow M$ which induces an isomorphism $M^{br}(X) \cong M(X)$ for all $X \in \Smk^{prop}$. This morphism $\mu^M$ plays a key role in the proof of Theorem \ref{Intro. str. A1}.

\subsection{Strictly $\aone$-invariant sheaves and birational invariance}

We denote $\Mod^{\aone}(\Lambda)$ (resp. $\Mod^{br}(\Lambda)$) for the category of strictly $\aone$-invariant (resp. birational) sheaves of $\Lambda$-modules. Then we have $\Mod^{br}(\Lambda) \subseteq \Mod^{\aone}(\Lambda)$. Indeed, all birational sheaves are $\aone$-invariant by Lemma \ref{A1-inv of br sh}. On the other hand, $\aone$-invariant birational sheaves of abelian groups are strictly $\aone$-invariant by \cite[proof of Lem. 2.4]{AH}. The following lemma says that every proper birational morphism induces an isomorphism for all strictly $\aone$-invariant sheaves.

\begin{lem}\label{proper birationality}
Let $M$ be a strictly $\aone$-invariant sheaf of $\Lambda$-modules and ${f : X \to Y}$ be a proper birational morphism of smooth $k$-varieties. Then the induced map $M(Y) \to M(X)$ is an isomorphism.
\end{lem}

\begin{proof}
We write $K = k(X)$ and identify $K$ with $k(Y)$ under the isomorphism induced by $f$. For each $U \in \Smk$, the map $M(U) \to M(k(U))$ gives an isomorphism
\begin{equation*}
M(U) \cong \bigcap_{x \in U^{(1)}} \Image (M(\mathcal{O}_{U,x}) \to M(k(U)))
\end{equation*}
by \cite[Lem. 4.2]{As}. Thus we only need to show that
\begin{equation*}
\bigcap_{x \in X^{(1)}} \Image (M(\mathcal{O}_{X,x}) \to M(K)) = \bigcap_{y \in Y^{(1)}} \Image (M(\mathcal{O}_{Y,y}) \to M(K)).
\end{equation*}
The inclusion $\supseteq$ follows from the commutative diagram
\begin{equation*}
\begin{CD}
M(Y) @>>> M(k(Y)) \\
@V{f^*}VV @VV{\cong}V \\
M(X) @>>> M(k(X)).
\end{CD}
\end{equation*}
We prove $\subseteq$. By the valuative criterion of properness, for every $y \in Y^{(1)}$ the commutative diagram
\begin{equation*}
\begin{CD}
\Spec K @>>> X \\
@VVV @VV{f}V \\
\Spec \mathcal{O}_{Y,y} @>>> Y
\end{CD}
\end{equation*}
has a lift $\phi_y : \Spec \mathcal{O}_{Y,y} \to X$. Then the local ring $\mathcal{O}_{X,\phi_y(y')}$ for the closed point $y' \in \Spec \mathcal{O}_{Y,y}$ coincides with the discrete valuation ring $\mathcal{O}_{Y,y}$ of $K$. In particular, $\phi_y(y')$ has codimension $1$ in $X$. Therefore, we have
\begin{align*}
\bigcap_{x \in X^{(1)}} \Image (M(\mathcal{O}_{X,x})  \to M(K)) &\subseteq \bigcap_{y \in Y^{(1)}} \Image (M(\mathcal{O}_{X,\phi_y(y')}) \to M(K)) \\
&= \bigcap_{y \in Y^{(1)}} \Image (M(\mathcal{O}_{Y,y}) \to M(K)).
\end{align*}
\end{proof}

\begin{rem}
More generally, Lemma \ref{proper birationality} holds for unramified sheaves of Morel \cite{Mo2} by the same proof.
\end{rem}

\subsection{Birational construction for strictly $\aone$-invariant sheaves}

Assume $k$ admits a resolution of singularities. We define a canonical functor ${\Mod^{\aone}(\Lambda) \to \Mod^{br}(\Lambda)}$. Lemma \ref{proper birationality} says that the restriction to $\Smk^{pv}$ of a strictly $\aone$-invariant sheaf induces a presheaf on $S_b^{-1}\Smk^{pv}$. Thus we have a functor
\begin{equation*}
\Mod^{\aone}(\Lambda) \to \mathcal{P}resh(S_b^{-1}\Smk^{pv},\Lambda).
\end{equation*}
Moreover, Lemma \ref{eq. br. sh.} and \eqref{eq. cat Smpv Smvar} give equivalences of categories
\begin{equation}\label{br modules preshe var}
\Mod^{br}(\Lambda) \xrightarrow{\cong} \mathcal{P}resh(S_b^{-1}\Smk^{var},\Lambda) \xrightarrow{\cong} \mathcal{P}resh(S_b^{-1}\Smk^{pv},\Lambda).
\end{equation}

\begin{defi}
When $k$ admits a resolution of singularities, we define a functor
\begin{equation*}
\Mod^{\aone}(\Lambda) \to \Mod^{br}(\Lambda);M \mapsto M^{br}
\end{equation*}
as the composition
\begin{equation*}
\Mod^{\aone}(\Lambda) \to \mathcal{P}resh(S_b^{-1}\Smk^{pv},\Lambda) \xrightarrow{\cong} \Mod^{br}(\Lambda).
\end{equation*}
\end{defi}

We construct a natural morphism of sheaves $\mu^M : M^{br} \to M$ for each $M \in \Mod^{\aone}(\Lambda)$. By the construction, we have $M^{br}(X) = M(X)$ for all ${X \in \Smk^{pv}}$. For every $U \in \Smk^{var}$, the map $M(U) \to M(k(U))$ induces a natural isomorphism
\begin{equation}\label{isom str aone nr}
M(U) \cong \bigcap_{x \in U^{(1)}} \Image(M(\mathcal{O}_{U,x}) \hookrightarrow M(k(U)))
\end{equation}
by \cite[Lem. 4.2]{As}. Thus \cite[Prop. 2.1.8]{CT} shows that
\begin{align*}
M^{br}(X) &\cong \bigcap_{x \in X^{(1)}} \Image(M(\mathcal{O}_{X,x}) \hookrightarrow M(k(X))) \\
&= \bigcap_{A \in \mathcal{P}(k(X)/k)} \Image(M(A) \hookrightarrow M(k(X))).
\end{align*}
For a smooth proper compactification $\overline{U}$ of $U \in \Smk^{var}$, we have
\begin{equation}\label{isom br nr in pf}
M^{br}(U) \cong M^{br}(\overline{U}) \cong \bigcap_{A \in \mathcal{P}(k(U)/k)} \Image(M(A) \hookrightarrow M(k(U))).
\end{equation}
Thus the sheaf $M^{br}$ can be regarded as the subsheaf
\begin{equation*}
U \mapsto \bigcap_{A \in \mathcal{P}(k(U)/k)} \Image(M(A) \hookrightarrow M(k(U)))
\end{equation*}
of $M$ under the isomorphism \eqref{isom str aone nr}. Then $\mu^M$ is defined as the embedding $M^{br} \hookrightarrow M$. We prove some basic properties of the functor $(-)^{br}$ and the morphism $\mu^M$.

\begin{prop}\label{adj. br functor}
Assume $k$ admits a resolution of singularities. Let $M$ be a strictly $\aone$-invariant sheaf. 
\begin{enumerate}
\item\label{muM inj} The morphism $\mu^M : M^{br} \to M$ is injective.
\item\label{muM iso} The induced map $\mu^M_X : M^{br}(X) \to M(X)$ is an isomorphism for every $X \in \Smk^{prop}$.
\item\label{muM ind Hom} The induced map
\begin{equation*}
\Hom_{\Mod^{\aone}(\Lambda)}(N,M^{br}) \xrightarrow{(\mu^M)_*} \Hom_{\Mod^{\aone}}(N,M)
\end{equation*}
is an isomorphism for every $N \in \Mod^{br}(\Lambda)$. In particular, the functor $(-)^{br}$ is right adjoint to the inclusion $\Mod^{br}(\Lambda) \hookrightarrow \Mod^{\aone}(\Lambda)$.
\end{enumerate} 
\end{prop}

\begin{proof}
The assertions \eqref{muM inj} and \eqref{muM iso} follow from the construction of $M^{br}$ and $\mu^M$. We prove \eqref{muM ind Hom}. Since $\mu^M$ is injective by \eqref{muM inj}, so is $(\mu^M)_*$. On the other hand, for a morphism $f : N \to M$ in $\Mod^{\aone}(\Lambda)$, we have a morphism $f' : N \to M^{br}$ defined by
\begin{align*}
N(U) &\cong \bigcap_{A \in \mathcal{P}(k(U)/k)} \Image(N(A) \hookrightarrow N(k(U))) \\
&\to \bigcap_{A \in \mathcal{P}(k(U)/k)} \Image(M(A) \hookrightarrow M(k(U))) \\
&\cong M^{br}(U)
\end{align*}
for each $U \in \Smk^{var}$ by using the isomorphism in \eqref{isom br nr in pf}. Then $\mu^M \circ f' = f$ and thus $(\mu^M)_*$ is surjective.
\end{proof}

\subsection{Birational construction for $\aone$-invariant sheaves with transfers}

Assume $k$ perfect. Let $\mathbf{HI}_k(\Lambda)$ be the category of $\aone$-invariant sheaves with transfers of $\Lambda$-modules. Note that every object in $\mathbf{HI}_k(\Lambda)$ is strictly $\aone$-invariant as a sheaf on $\Smk$ (see \cite[Thm. 13.8]{MVW}). Lemma \ref{A1-inv of br sh} shows that $\mathbf{NST}_k^{br}(\Lambda)$ is a full subcategory of $\mathbf{HI}_k(\Lambda)$. In \cite[\S6]{KS3}, Kahn-Sujatha construct a right adjoint functor
\begin{equation*}
\mathbf{NST}_k^{br}(\Lambda) \to \mathbf{HI}_k(\Lambda); M \mapsto M_{nr}
\end{equation*}
of the inclusion $\mathbf{NST}_k^{br}(\Lambda) \hookrightarrow \mathbf{HI}_k(\Lambda)$ which sends $M \in \mathbf{HI}_k(\Lambda)$ to the presheaf with transfers defined by
\begin{equation*}
U \mapsto \bigcap_{A \in \mathcal{P}(k(U)/k)} \Ker \left( M(k(U)) \xrightarrow{\partial_{A}} M_{-1}(\kappa_A) \right).
\end{equation*}
Here $M_{-1}$ is the contraction of $M$ and $\partial_{A}$ is the residue map associated with $A$ (see definitions \cite[Lec. 23 and Ex. 24.6]{MVW}). On the other hand, the sequence
\begin{equation*}
0 \to M(A) \to M(k(U)) \xrightarrow{\partial_A} M_{-1}(\kappa_A) \to 0
\end{equation*}
is exact by \cite[Ex. 24.6]{MVW}. Thus we have
\begin{equation}\label{eq. Mnr(U) isom Mbr(U)}
M_{nr}(U) = \bigcap_{A \in \mathcal{P}(k(U)/k)} \Image (M(A) \to M(k(U))).
\end{equation}
Then \cite[Lem. 6.2.3, 6.2.4 and 6.2.6]{KS3} show that $M_{nr}$ can be regarded as the subsheaf
\begin{equation*}
U \mapsto \bigcap_{A \in \mathcal{P}(k(U)/k)} \Image (M(A) \to M(k(U)))
\end{equation*}
of $M$ under the isomorphism \eqref{isom str aone nr}. Thus when $k$ admits a resolution of singularities, there exists a canonical isomorphism $\Gamma_*(M_{nr}) \cong (\Gamma_*M)^{br}$ in $\Mod(\Lambda)$, \textit{i.e.}, the diagram
\begin{equation*}
\begin{CD}
\mathbf{HI}_k(\Lambda) @>{(-)_{nr}}>> \mathbf{NST}_k^{br}(\Lambda) \\
@V{\Gamma_*}VV @VV{\Gamma_*}V \\
\Mod^{\aone}(\Lambda) @>{(-)^{br}}>> \Mod^{br}(\Lambda)
\end{CD}
\end{equation*}
is $2$-commutative, by the construction of $M^{br}$. Moreover, if we write $\nu^M$ for the embedding $M_{nr} \hookrightarrow M$, the diagram
\begin{equation*}
\begin{CD}
\Gamma_*(M_{nr}) @>{\Gamma_*(\nu^M)}>> \Gamma_*M \\
@V{\cong}VV @| \\
(\Gamma_*M)_{nr} @>{\mu^M}>> \Gamma_*M
\end{CD}
\end{equation*}
commutes. We obtain an analogue of Proposition \ref{adj. br functor} for the functor $(-)_{nr}$ without assuming a resolution of singularities.

\begin{prop}\label{Mnr isom M with tr}
Assume $k$ perfect. Let $M$ be an $\aone$-invariant sheaf with transfers.
\begin{enumerate}
\item\label{Prop nu inj} The map $\nu^M : M_{nr} \to M$ is injective.
\item\label{Prop nu isom} The induced map $\nu_X^{M} : M_{nr}(X) \to M(X)$ is an isomorphism.
\end{enumerate}
\end{prop}

\begin{proof}
The assertion \eqref{Prop nu inj} follows from the construction of $\nu^M$. By \eqref{eq. Mnr(U) isom Mbr(U)} and \cite[Prop. 2.1.8]{CT}, we have
\begin{align*}
M_{nr}(X) &\cong \bigcap_{A \in \mathcal{P}(k(X)/k)} \Image (M(A) \to M(k(X))) \\
&= \bigcap_{x \in X^{(1)}} \Image(M(\mathcal{O}_{X,x}) \hookrightarrow M(k(X))).
\end{align*}
Thus \eqref{Prop nu isom} follows from \eqref{isom str aone nr}.
\end{proof}

\section{Structure of zeroth $\aone$- and Suslin homology}\label{ST}

\subsection{Structure theorem of zeroth $\aone$-homology}

Our aim of this subsection is to prove a structure theorem of the $\aone$-homology of smooth proper varieties. In \cite{Mo1}, Morel defined the $\aone$-homology sheaf $\HAone{i}(\mathscr{F})$ for $\mathscr{F} \in \Shvk$ as a strictly $\aone$-invariant sheaf. Especially, the zeroth $\aone$-homology functor
\begin{equation*}
\HAone{0}(-;\Lambda) : \Shvk \to \Mod^{\aone}(\Lambda)
\end{equation*}
can be characterized as a left adjoint of the forgetful functor $\Mod^{\aone}(\Lambda) \to \Shvk$ (see \cite[Lem. 3.3]{As}). For $\mathscr{F} \in \Shvk$, we denote $\Lambda_{pre}(\mathscr{F})$ for the presheaf of free $\Lambda$-modules generated by $\mathscr{F}$. Note that if $\mathscr{F}$ is birational, then so is $\Lambda_{pre}(\mathscr{F})$. Thus we have the functor
\begin{equation*}
\Lambda_{pre} : \Shvk^{br} \to \Mod^{br}(\Lambda).
\end{equation*}
This is left adjoint to the forgetful functor $\Mod^{br}(\Lambda) \to \Shvk^{br}$ by the definition. Our structure theorem of zeroth $\aone$-homology is stated as follows.

\begin{thm}\label{str H0A1}
Assume $k$ admits a resolution of singularities. For every ${X \in \Smk}$, there exists a natural epimorphism of sheaves
\begin{equation*}
\Lambda_{pre}(\pi_0^{b\aone}(X)) \twoheadrightarrow \HAone{0}(X;\Lambda).
\end{equation*}
Moreover, this is an isomorphism if $X$ is proper.
\end{thm}

\begin{proof}
By Lemmas \ref{adj pi b aone} and \ref{adj. br functor}, we obtain three adjunctions
\begin{equation*}
\Shvk \rightleftarrows \Shvk^{br} \rightleftarrows \Mod^{br}(\Lambda) \rightleftarrows \Mod^{\aone}(\Lambda).
\end{equation*}
Hence the composite functor
\begin{equation*}
\Shvk \xrightarrow{\pi_0^{b\aone}} \Shvk^{br} \xrightarrow{\Lambda_{pre}} \Mod^{br}(\Lambda) \hookrightarrow \Mod^{\aone}(\Lambda)
\end{equation*}
is left adjoint to the composition
\begin{equation*}
\Shvk \hookleftarrow \Shvk^{br} \leftarrow \Mod^{br}(\Lambda) \xleftarrow{(-)^{br}} \Mod^{\aone}(\Lambda).
\end{equation*}
By using this adjunction and Yoneda's lemma in $\Smk$, we have a natural isomorphism
\begin{equation}\label{isom Hom Mbr}
\Hom_{\Mod^{\aone}(\Lambda)}(\Lambda_{pre}(\pi_0^{b\aone}(X)),M) \cong \Hom_{\Shvk}(X,M^{br}) \cong M^{br}(X)
\end{equation}
for all $X \in \Smk$ and all $M \in \Mod^{\aone}(\Lambda)$. On the other hand, the zeroth $\aone$-homology functor
\begin{equation*}
\HAone{0}(-;\Lambda) : \Shvk \to \Mod^{\aone}(\Lambda)
\end{equation*}
is left adjoint to the forgetful functor $\Mod^{\aone}(\Lambda) \to \Shvk$. Therefore, we also obtain a natural isomorphism
\begin{equation}\label{isom Hom M}
\Hom_{\Mod^{\aone}(\Lambda)}(\HAone{0}(X;\Lambda),M) \cong \Hom_{\Shvk}(X,M) \cong M(X).
\end{equation}
By \eqref{isom Hom Mbr} and \eqref{isom Hom M}, the natural monomorphism $\mu^M : M^{br} \to M$ induces a injection
\begin{equation*}
\Hom_{\Mod^{\aone}(\Lambda)}(\Lambda_{pre}(\pi_0^{b\aone}(X)),M) \to \Hom_{\Mod^{\aone}(\Lambda)}(\HAone{0}(X;\Lambda),M).
\end{equation*}
This map is an isomorphism when $X$ is proper by Proposition \ref{adj. br functor}. Thus Yoneda's lemma in $\Mod^{\aone}(\Lambda)$ leads an epimorphism of sheaves
\begin{equation*}
\Lambda_{pre}(\pi_0^{b\aone}(X)) \twoheadrightarrow \HAone{0}(X;\Lambda)
\end{equation*}
which is an isomorphism for proper $X$.
\end{proof}

\begin{cor}\label{free R-mod of R-eq}
Assume $k$ admits a resolution of singularities. For $X \in \Smk^{prop}$ and $U \in \Smk^{var}$, the section $\HAone{0}(X;\Lambda)(U)$ is the free $\Lambda$-module generated by $X(k(U))/R$.
\end{cor}

\begin{proof}
This follows from Theorem \ref{str H0A1} and Lemma \ref{Req and pibr}.
\end{proof}

We obtain the K\"unneth formula and the universal coefficient theorem of $\aone$-homology in degree zero.

\begin{cor}
Assume $k$ admits a resolution of singularities. For ${X, Y \in \Smk^{prop}}$, there exists a natural isomorphism
\begin{equation*}
\HAone{0}(X \times Y;\Lambda) \cong \HAone{0}(X;\Lambda) \otimes_\Lambda \HAone{0}(Y;\Lambda),
\end{equation*}
where $\otimes_\Lambda$ means the tensor product of presheaves over $\Lambda$.
\end{cor}

\begin{proof}
This follows from Theorem \ref{str H0A1} and the natural bijection
\begin{equation*}
(X \times Y)(K)/R \cong (X(K)/R) \times (Y(K)/R)
\end{equation*}
for each $K \in \mathcal{F}_k$.
\end{proof}

\begin{cor}
Assume $k$ admits a resolution of singularities. For ${X \in \Smk^{prop}}$, there exists a natural isomorphism
\begin{equation*}
\HAone{0}(X;\Lambda) \cong \HAone{0}(X;\mathbb{Z}) \otimes \Lambda,
\end{equation*}
where $\otimes$ means the tensor product of presheaves over $\mathbb{Z}$.
\end{cor}

\begin{proof}
This follows from $\Lambda_{pre}(-) \cong \mathbb{Z}_{pre}(-) \otimes \Lambda$ and Theorem \ref{str H0A1}.
\end{proof}

\begin{rem}\label{rmk pairing of A1 homology}
Assume $k$ admits a resolution of singularities. By Theorem \ref{str H0A1}, there exists a pairing
\begin{equation*}
\HAone{0}(X;\Lambda)(Y) \times \HAone{0}(Y;\Lambda)(Z) \to \HAone{0}(X;\Lambda)(Z)
\end{equation*}
defined by the map
\begin{equation*}
\Hom_{S_b^{-1}\Smk}(Y,X) \times \Hom_{S_b^{-1}\Smk}(Z,Y) \to \Hom_{S_b^{-1}\Smk}(Z,X); (f,g) \mapsto f \circ g
\end{equation*}
for all $X,Y,Z \in \Smk^{prop}$. In particular, this pairing gives a ring structure for $\HAone{0}(X;\Lambda)(X)$.
\end{rem}

\subsection{Rational points and finiteness}

In this subsection, we give some examples where $\HAone{0}(X;\Lambda)(k)$ is finitely generated. We first introduce the following notation.

\begin{defi}
We write $b_0^{\aone}(X) = \dim_{\mathbb{Q}} \HAone{0}(X;\mathbb{Q})(k)$ for $X \in \Smk$. This is called \textit{the zeroth $\aone$-Betti number of $X$}.
\end{defi}

By Corollary \ref{free R-mod of R-eq}, $b_0^{\aone}(X)$ coincides with the rank of the free $\Lambda$-module $\HAone{0}(X;\Lambda)(k)$ if $X$ is proper. Our structure theorem of zeroth $\aone$-homology enables to relate the $\aone$-Betti number to $k$-rational points modulo $R$-equivalence.

\begin{thm}\label{A1-betti and rat pt}
Assume $k$ admits a resolution of singularities. For a smooth proper $k$-variety $X$, we have 
\begin{equation*}
b_0^{\aone}(X) = \#(X(k)/R).
\end{equation*}
In particular, $b_0^{\aone}(X) = 0$ if and only if $X(k) = \emptyset$.
\end{thm}

\begin{proof}
This follows from Corollary \ref{free R-mod of R-eq}.
\end{proof}

The following proposition implies the finiteness of $b_0^{\aone}(X)$ for a rationally connected smooth proper real variety $X$.

\begin{prop}
Let $X$ be a rationally connected smooth proper variety over $\mathbb{R}$. Then we have
\begin{equation*}
\HAone{0}(X;\Lambda)(\mathbb{R}) \cong H_0(X(\mathbb{R});\Lambda).
\end{equation*}
In particular, $b_0^{\aone}(X)$ coincides with the ordinary zeroth Betti number of the real manifold $X(\mathbb{R})$.
\end{prop}

\begin{proof}
By \cite[Cor. 1.7]{Ko}, the set of $R$-equivalence classes $X(\mathbb{R})/R$ coincides with the connected components of $X(\mathbb{R})$. Thus the assertion follows from Corollary \ref{free R-mod of R-eq}.
\end{proof}

By using Theorem \ref{A1-betti and rat pt}, we also obtain some examples where $b_0^{\aone}(X)$ is finite.

\begin{examp}
\begin{enumerate}
\item Let $X$ be a smooth proper variety with finitely many $k$-rational points. Then $b_0^{\aone}(X)$ is finite. Thus for example, $b_0^{\aone}(C)$ is finite for a smooth projective curve $C$ of genus $\geq 2$ over a number field by Mordell conjecture proved by Faltings \cite{Fa}.
\item Let $X$ be a rationally connected smooth proper variety over either of $\mathbb{C}$ or a $p$-adic field. Then $b_0^{\aone}(X)$ is also finite by \cite[Cor. 1.5]{Ko}.
\end{enumerate}
\end{examp}

\begin{rem}\label{other expressions of b0A1}
For $X$ a smooth $k$-variety, $\HAone{0}(X;\mathbb{Z})(k)$ can be expressed in terms of triangulated categories as follows.
\begin{enumerate}
\item There exist natural isomorphisms
\begin{equation*}
\HAone{0}(X;\mathbb{Z})(k) \cong \Hom_{\Mod(\mathbb{Z})}(\mathbb{Z},\HAone{0}(X;\mathbb{Z})) \cong \Hom_{D_{\aone}(k)}(\mathbb{Z},\mathbb{Z}(X)).
\end{equation*}
Here $D_{\aone}(k)$ is the $\aone$-derived category defined by Morel \cite{Mo1} and $\mathbb{Z}(X)$ is the Nisnevich sheafification of $\mathbb{Z}_{pre}(X)$. This follows from \cite[Lem. 3.3]{As} and its proof.
\item There also exist natural isomorphisms
\begin{equation*}
\HAone{0}(X;\mathbb{Z})(k) \cong \pi_0^{S}(\Sigma_{S^1}^{\infty}(X_+))(k) \cong \Hom_{\mathcal{SH}_{S^1}(k)}(S^0,\Sigma_{S^1}^{\infty}(X_+)).
\end{equation*}
Here $\Sigma_{S^1}^{\infty}(X_+)$ is the infinity $S^1$-suspension of $X_+ = X \sqcup \Spec k$, $\pi_0^{S}(-)$ is the $S^1$-stable zeroth $\aone$-homotopy sheaf, $\mathcal{SH}_{S^1}(k)$ is the $\aone$-homotopy category of $S^1$-spectra, and $S^0$ is the sphere spectrum (see definitions \cite{Mo1}). The first isomorphism is given by \cite[Prop. 2.1]{AH} and the second isomorphism is given by the definition of $\pi_0^{S}(-)$.
\end{enumerate}
\end{rem}

\subsection{Zeroth Suslin homology and zero cycles}

In this subsection, we prove the Suslin homology version of Theorem \ref{str H0A1}. For ${X \in \Smk}$, the Suslin homology sheaf $\mathbf{H}_i^S(X;\Lambda)$ is defined as the $i$-th homology sheaf with transfers of the motive of $X$ with $\Lambda$-coefficients in the sense of Voevodsky \cite{Vo00}. Then the section $\mathbf{H}_i^S(X;\Lambda)(k)$ coincides with the singular homology of Suslin-Voevodsky \cite{SV}, called Suslin homology group. Note that $\mathbf{H}_i^S(X;\Lambda)$ is $\aone$-invariant by the $\aone$-invariance of the motive of $X$. Thus we have a functor
\begin{equation*}
\mathbf{H}_0^S(-;\Lambda) : \mathcal{C}or_{k,\Lambda} \to \mathbf{HI}_k(\Lambda).
\end{equation*}
As a counterpart of Theorem \ref{str H0A1}, we construct an isomorphism of sheaves with transfers $\Lambda\pi_{0,tr}^{b\aone}(X) \cong \mathbf{H}_0^S(X;\Lambda)$ for $X \in \Smk^{prop}$.

\begin{thm}\label{str thm of Suslin}
Assume $k$ perfect. For every $X \in \Smk$, there exists a natural epimorphism of sheaves with transfers
\begin{equation*}
\Lambda\pi_{0,tr}^{b\aone}(X) \twoheadrightarrow \mathbf{H}_0^S(X;\Lambda).
\end{equation*}
Moreover, this is an isomorphism if $X$ is proper.
\end{thm}

\begin{proof}
By Lemma \ref{adj pi b aone tr}, the composite functor
\begin{equation*}
\mathbf{NST}_k(\Lambda) \xrightarrow{\Lambda\pi_{0,tr}^{b\aone}} \mathbf{NST}_k^{br}(\Lambda) \hookrightarrow \mathbf{HI}_k(\Lambda)
\end{equation*}
is left adjoint to the functor $(-)_{nr} : \mathbf{HI}_k(\Lambda) \to \mathbf{NST}_k(\Lambda)$. For all $X \in \Smk$ and all $M \in \mathbf{HI}_k(\Lambda)$, thus we obtain natural isomorphisms
\begin{equation*}
\Hom_{\mathbf{HI}_k(\Lambda)}(\Lambda\pi_{0,tr}^{b\aone}(X),M) \cong \Hom_{\mathbf{NST}_k(\Lambda)}(\Lambda_{tr}(X),M_{nr}) \cong M_{nr}(X)
\end{equation*}
by Yoneda's lemma in $\mathcal{C}or_{k,\Lambda}$. On the other hand, \cite[Lem. 3.3]{As} gives an isomorphism
\begin{equation*}
\Hom_{\mathbf{HI}_k(\Lambda)}(\mathbf{H}_0^S(X;\Lambda),M) \cong M(X).
\end{equation*}
Thus the map
\begin{equation*}
\Hom_{\mathbf{HI}_k(\Lambda)}(\Lambda\pi_{0,tr}^{b\aone}(X),M_{nr}) \to \Hom_{\mathbf{HI}_k(\Lambda)}(\mathbf{H}_0^S(X;\Lambda),M)
\end{equation*}
obtained by $\nu^M : M_{nr} \to M$ introduces an epimorphism of sheaves with transfers
\begin{equation*}
\Lambda\pi_{0,tr}^{b\aone}(X) \to \mathbf{H}_0^S(X;\Lambda)
\end{equation*}
by Yoneda's lemma in $\mathbf{HI}_k(\Lambda)$. By Proposition \ref{Mnr isom M with tr}, this morphism is an isomorphism if $X$ is proper.
\end{proof}

Assume $k$ perfect. For every $X \in \Smk^{prop}$, there exists a canonical isomorphism
\begin{equation*}
\mathbf{H}_0^S(X;\mathbb{Z})(K) \cong \mathrm{CH}_0(X_{K})
\end{equation*}
for all $K \in \mathcal{F}_k$ (see \cite[Ex. 4.9]{As}). Thus we obtain an application to Chow groups of zero cycles. We simply write 
\begin{equation*}
S_b^{-1}\mathcal{C}or_{k}(X,Y) = \Hom_{S_b^{-1}\mathcal{C}or_{k,\mathbb{Z}}}(X,Y)
\end{equation*}
for $X,Y \in \Smk$.

\begin{cor}\label{cor Sb-1Cor isom CH0}
Assume $k$ perfect. Let $X$ be a smooth proper $k$-variety and $U$ be a smooth $k$-variety. Then we have
\begin{equation*}
S_b^{-1}\mathcal{C}or_{k}(U,X) \cong \mathrm{CH}_0(X_{k(U)}).
\end{equation*}
\end{cor}

\begin{proof}
By Theorem \ref{str thm of Suslin} and the construction of $\Lambda\pi_{0,tr}^{b\aone}(X)$, we have
\begin{equation*}
S_b^{-1}\mathcal{C}or_{k}(U,X) \cong \mathbb{Z}\pi_{0,tr}^{b\aone}(X)(U) \cong \mathbf{H}_0^S(X;\mathbb{Z})(U) \cong \mathbf{H}_0^S(X;\mathbb{Z})(k(U)).
\end{equation*}
On the other hand, we also have $\mathbf{H}_0^S(X;\mathbb{Z})(k(U)) \cong \mathrm{CH}_0(X_{k(U)})$ by \cite[Ex. 4.9]{As}.
\end{proof}

\begin{rem}
Assume $k$ perfect. Like Remark \ref{rmk pairing of A1 homology}, the composition of morphisms in $S_b^{-1}\mathcal{C}or_{k,\Lambda}$ also induces a pairing
\begin{equation*}
\mathbf{H}_0^S(X;\Lambda)(Y) \times \mathbf{H}_0^S(Y;\Lambda)(Z) \to \mathbf{H}_0^S(X;\Lambda)(Z)
\end{equation*}
and a ring structure of $\mathbf{H}_0^S(X;\Lambda)(X)$ for all $X,Y,Z \in \Smk^{pv}$. Thus we obtain a pairing
\begin{equation*}
\mathrm{CH}_0(X_{k(Y)}) \times \mathrm{CH}_0(Y_{k(Z)}) \to \mathrm{CH}_0(X_{k(Z)})
\end{equation*}
and a ring structure of $\mathrm{CH}_0(X_{k(X)})$.
\end{rem}

\section{Universal birational invariance}\label{UBI}

\subsection{Universal birational invariance of $\aone$-homology}\label{subsec univ birat A1}

Our aim of this subsection is to prove the main theorem of this paper, \textit{i.e.}, the universal birational invariance of zeroth $\aone$-homology. For a category $\mathcal{C}$, we denote $\Lambda(\mathcal{C})$ for the category whose objects are same with $\mathcal{C}$ and that $\Hom_{\Lambda(\mathcal{C})}(A,B)$ for each $A,B \in \Lambda(\mathcal{C})$ is the free $\Lambda$-module generated by $\Hom_{\mathcal{C}}(A,B)$. Note that every functor from $\mathcal{C}$ to an $\Lambda$-enriched category uniquely factors through the canonical $\Lambda$-linear functor $\mathcal{C} \to \Lambda(\mathcal{C})$. Thus we have a functor
\begin{equation}\label{zeroth aone homology RSb-1}
\HAone{0}(-;\Lambda) : \Lambda(S_b^{-1}\Smk^{prop}) \to \Mod(\Lambda).
\end{equation}

\begin{prop}\label{A1homology is fully faithful}
Assume $k$ admits a resolution of singularities. Then the functor \eqref{zeroth aone homology RSb-1} is fully faithful.
\end{prop}

\begin{proof}
By \cite[Thm. 6.4]{KS1}, we only need to show that the functor
\begin{equation}\label{univ. birat. of H0A1 equiv}
\HAone{0}(-;\Lambda) : \Lambda(S_b^{-1}\Smk^{pv}) \to \Mod(\Lambda)
\end{equation}
is fully faithful. The canonical functor $S_b^{-1}\Smk^{pv} \to \Lambda(S_b^{-1}\Smk^{pv})$ induces an equivalence of categories
\begin{equation*}
\mathcal{P}resh(S_b^{-1}\Smk^{pv},\Lambda) \xrightarrow{\cong} \mathcal{L}in(\Lambda(S_b^{-1}\Smk^{pv}),\Lambda).
\end{equation*}
On the other hand, we obtain equivalences
\begin{equation*}
\Mod^{br}(\Lambda) \xrightarrow{\cong} \mathcal{P}resh(S_b^{-1}\Smk^{var},\Lambda) \xrightarrow{\cong} \mathcal{P}resh(S_b^{-1}\Smk^{pv},\Lambda)
\end{equation*}
by \eqref{br modules preshe var}. Thus we have a functor
\begin{equation*}
\mathscr{H} : \Lambda(S_b^{-1}\Smk^{pv}) \hookrightarrow \mathcal{L}in(\Lambda(S_b^{-1}\Smk^{pv}),\Lambda) \cong \Mod^{br}(\Lambda) \hookrightarrow \Mod(\Lambda),
\end{equation*}
where the first functor is the Yoneda embedding. Therefore, $\mathscr{H}$ is fully faithful by Yoneda's lemma in $\Lambda(S_b^{-1}\Smk^{pv})$. By Theorem \ref{str H0A1}, there exists a natural isomorphism
\begin{equation*}
\mathscr{H}(X) \cong \Hom_{\Lambda(S_b^{-1}\Smk^{pv})}(-,X) \cong \Lambda_{pre}(\pi_0^{b\aone}(X)) \cong \HAone{0}(X;\Lambda)
\end{equation*}
for all $X \in \Smk^{pv}$. Thus the diagram
\begin{equation*}
\xymatrix{
\Smk^{pv} \ar[r]^-{\HAone{0}(-;\Lambda)} \ar[d] &\Mod(\Lambda) \ar@{=}[d]\\
\Lambda(S_b^{-1}\Smk^{pv}) \ar[r]^-{\mathscr{H}} &\Mod(\Lambda)
}
\end{equation*}
is $2$-commutative. This diagram shows that the functor \eqref{univ. birat. of H0A1 equiv} is naturally equivalent to $\mathscr{H}$. Since $\mathscr{H}$ is fully faithful, so is \eqref{univ. birat. of H0A1 equiv}.
\end{proof}

Proposition \ref{A1homology is fully faithful} shows that the zeroth $\aone$-homology functor on $S_b^{-1}\Smk^{prop}$ is conservative and faithful.

\begin{cor}\label{cons. and faithful}
Assume $k$ admits a resolution of singularities. Then the functor
\begin{equation}\label{H0A1 on brSmkpv}
\HAone{0}(-;\Lambda) : S_b^{-1}\Smk^{prop} \to \Mod(\Lambda).
\end{equation}
is conservative and faithful.
\end{cor}

\begin{proof}
Since the functors $S_b^{-1}\Smk^{prop} \to \Lambda(S_b^{-1}\Smk^{prop})$ and \eqref{zeroth aone homology RSb-1} are conservative and faithful by Proposition \ref{A1homology is fully faithful}, so is \eqref{H0A1 on brSmkpv}.
\end{proof}

For describing universal birational invariance, we introduce the following term.

\begin{defi}\label{def. br. ff}
Let $\mathcal{C}$ be a full subcategory of $\Smk^{prop}$. Then $\mathcal{C}$ is called \textit{a birationally fully faithful subcategory of $\Smk^{prop}$}, if the functor ${S_b^{-1}\mathcal{C} \to S_b^{-1}\Smk^{prop}}$ is fully faithful.
\end{defi}

We see some basic examples of birationally fully faithful subcategories of $\Smk^{prop}$.

\begin{examp}\label{ex. br. ff}
Clearly, $\Smk^{prop}$ is a birationally fully faithful subcategory of itself. Let $\Smk^{pjv}$ (resp. $\Smk^{proj}$) be the full subcategory of $\Smk^{prop}$ spanned by smooth projective $k$-varieties (resp. $k$-schemes). When $k$ admits a resolution of singularities, $\Smk^{proj}$ is a birationally fully faithful subcategory of $\Smk^{prop}$ by the equivalence $\Smk^{proj} \cong \Smk^{prop}$ in \cite[Thm. 8.8]{KS1}. Moreover, $\Smk^{pv}$ and $\Smk^{pjv}$ are also birationally fully faithful subcategories of $\Smk^{prop}$ by \cite[Thm. 6.4]{KS1}.
\end{examp}

We prove the universal birational invariance of the $\aone$-homology functor on birationally fully faithful subcategories of $\Smk^{prop}$. For a subcategory $\mathcal{C} \subseteq \Smk^{prop}$, we denote $\mathbf{Im}^{\mathcal{C}}_\Lambda\HAone{0}$ for the full subcategory of $\Mod(\Lambda)$ spanned by sheaves isomorphic to $\HAone{0}(X;\Lambda)$ for some $X \in \mathcal{C}$.

\begin{thm}\label{univ. birat. of H0A1}
Assume $k$ admits a resolution of singularities. Let $\mathcal{C}$ be a birationally full subcategory of $\Smk^{prop}$ and $\mathcal{A}$ be an arbitrary category enriched by $\Lambda$-modules.
\begin{enumerate}
\item\label{univ. birat. covariant} Let $F : \mathcal{C} \to \mathcal{A}$ be an arbitrary functor which sends each birational morphism to an isomorphism. Then there exists one and only one (up to a natural equivalence) $\Lambda$-linear functor $F_{S_b} : \mathbf{Im}^{\mathcal{C}}_\Lambda\HAone{0} \to \mathcal{A}$ such that the diagram
\begin{equation*}
\xymatrix{
\mathcal{C} \ar[r]^-{F} \ar[d] &\mathcal{A} \\
\mathbf{Im}^{\mathcal{C}}_\Lambda\HAone{0} \ar@{.>}[ru]_-{F_{S_b}}
}
\end{equation*}
is $2$-commutative.
\item\label{univ. birat. contravariant} Let $F' : \mathcal{C}^{op} \to \mathcal{A}$ be an arbitrary functor which sends each birational morphism to an isomorphism. Then there exists one and only one (up to a natural equivalence) $\Lambda$-linear functor $F'_{S_b} : (\mathbf{Im}^{\mathcal{C}}_\Lambda\HAone{0})^{op} \to \mathcal{A}$ such that the diagram
\begin{equation*}
\xymatrix{
\mathcal{C}^{op} \ar[r]^-{F'} \ar[d] &\mathcal{A} \\
(\mathbf{Im}^{\mathcal{C}}_\Lambda\HAone{0})^{op} \ar@{.>}[ru]_-{F'_{S_b}}
}
\end{equation*}
is $2$-commutative.
\end{enumerate}
\end{thm}

\begin{proof}
Since $S_b^{-1}\mathcal{C} \to S_b^{-1}\Smk^{prop}$ is fully faithful, so is the composition
\begin{equation}\label{unv birat eq 1}
\Lambda(S_b^{-1}\mathcal{C}) \to \Lambda(S_b^{-1}\Smk^{prop}) \to \Mod(\Lambda)
\end{equation}
by Proposition \ref{A1homology is fully faithful}. On the other hand, the essential image of the functor \eqref{unv birat eq 1} coincides with $\mathbf{Im}^{\mathcal{C}}_\Lambda\HAone{0}$. Thus we have an equivalence $\Lambda(S_b^{-1}\mathcal{C}) \cong \mathbf{Im}^{\mathcal{C}}_\Lambda\HAone{0}$ such that the diagram
\begin{equation}\label{diagram univ. covariant}
\begin{CD}
\mathcal{C} @= \mathcal{C} \\
@VVV @VV{\HAone{0}(-;\Lambda)}V \\
\Lambda(S_b^{-1}\mathcal{C}) @>\cong>> \mathbf{Im}^{\mathcal{C}}_\Lambda\HAone{0}
\end{CD}
\end{equation}
is $2$-commutative. Thus \eqref{univ. birat. covariant} follows from the universality on $\mathcal{C} \to \Lambda(S_b^{-1}\mathcal{C})$. Similarly, we obtain \eqref{univ. birat. contravariant} by taking the opposite categories in \eqref{diagram univ. covariant}.
\end{proof}

\begin{rem}\label{rmk H0A1 ext Modbr A}
The functor $F_{S_b}$ (resp. $F'_{S_b}$) in Theorem \ref{univ. birat. of H0A1} is canonically extended to ${\Mod^{br}(\Lambda) \to \mathcal{A}}$ (resp. ${(\Mod^{br}(\Lambda))^{op} \to \mathcal{A}}$), if $\Smk^{pjv} \subseteq \mathcal{C}$ and $\mathcal{A}$ is cocomplete. Indeed, the equivalences \eqref{br modules preshe var} and $\Smk^{pjv} \cong \Smk^{var}$ in \cite[Prop. 8.5]{KS1} induce
\begin{equation*}
\Mod^{br}(\Lambda) \cong \mathcal{L}in(\Lambda(S_b^{-1}\Smk^{var}),\Lambda) \cong \mathcal{L}in(\Lambda(S_b^{-1}\Smk^{pjv}),\Lambda).
\end{equation*}
Under the identification by this equivalence, the category $\mathbf{Im}^{\mathcal{C}}_\Lambda\HAone{0}$ contains the full subcategory of $\Mod^{br}(\Lambda)$ consisting of representable presheaves on $\Lambda(S_b^{-1}\Smk^{pjv})$ by Theorem \ref{str H0A1}. Thus the purpose extension is obtained as a left Kan extension of $\Lambda(S_b^{-1}\Smk^{pjv}) \to \mathcal{A}$ (resp. $(\Lambda(S_b^{-1}\Smk^{pjv}))^{op} \to \mathcal{A}$).
\end{rem}

\subsection{Universal birational invariance of Suslin homology}

In this subsection we prove the Suslin homology version of the universal birational invariance property. We have a functor
\begin{equation}\label{Suslin homology from Sb-1Cor}
\mathbf{H}_0^S(-;\Lambda) : S_b^{-1}\mathcal{C}or_{k,\Lambda}^{prop} \to \mathbf{NST}_k(\Lambda).
\end{equation}

\begin{prop}\label{prop Suslin fully faithful}
Assume $k$ perfect and admits a resolution of singularities. Then the functor \eqref{Suslin homology from Sb-1Cor} is fully faithful.
\end{prop}

\begin{proof}
By \cite[Thm. 6.4]{KS1}, we only need to show that the functor
\begin{equation}\label{eq H0S Sb-1Corpv}
\mathbf{H}_0^S(-;\Lambda) : S_b^{-1}\mathcal{C}or_{k,\Lambda}^{pv} \to \mathbf{NST}_k(\Lambda)
\end{equation}
is fully faithful. By Lemma \ref{eq. br. sh. cor} and Proposition \ref{corpv eq corvar}, we obtain equivalences of categories
\begin{equation*}
\mathbf{NST}_k^{br}(\Lambda) \xleftarrow{\cong} \mathcal{L}in(S_b^{-1}\mathcal{C}or_{k,\Lambda}^{var},\Lambda) \xrightarrow{\cong} \mathcal{L}in(S_b^{-1}\mathcal{C}or_{k,\Lambda}^{pv},\Lambda).
\end{equation*}
Thus we have a functor
\begin{equation*}
\mathscr{H}_{tr} : S_b^{-1}\mathcal{C}or_{k,\Lambda}^{pv} \hookrightarrow \mathcal{L}in(S_b^{-1}\mathcal{C}or_{k,\Lambda}^{pv},\Lambda) \cong \mathbf{NST}_k^{br}(\Lambda) \hookrightarrow \mathbf{NST}_k(\Lambda)
\end{equation*}
where the first functor is the Yoneda embedding. Hence $\mathscr{H}_{tr}$ is fully faithful by Yoneda's lemma in $S_b^{-1}\mathcal{C}or_{k,\Lambda}^{pv}$. By Theorem \ref{str thm of Suslin}, there exists a natural isomorphism
\begin{equation*}
\mathscr{H}_{tr}(X) \cong S_b^{-1}\mathcal{C}or_{k,\Lambda}(-,X) \cong \Lambda\pi_{0,tr}^{b\aone}(X) \cong \mathbf{H}_0^S(X;\Lambda)
\end{equation*}
for all $X \in \mathcal{C}or_{k,\Lambda}^{pv}$. Thus the diagram
\begin{equation*}
\begin{CD}
\mathcal{C}or_{k,\Lambda}^{pv} @>{\mathbf{H}_0^S(-;\Lambda)}>>  \mathbf{NST}_k(\Lambda) \\
@VVV @| \\
S_b^{-1}\mathcal{C}or_{k,\Lambda}^{pv} @>{\mathscr{H}_{tr}}>>  \mathbf{NST}_k(\Lambda)
\end{CD}
\end{equation*}
is $2$-commutative. Therefore, since $\mathscr{H}_{tr}$ is fully faithful, so is the functor \eqref{eq H0S Sb-1Corpv}.
\end{proof}

We also introduce birationally fully faithful subcategories of $\mathcal{C}or_{k,\Lambda}^{prop}$.

\begin{defi}
Let $\mathcal{C}$ be a full subcategory of $\mathcal{C}or_{k,\Lambda}^{prop}$. Then $\mathcal{C}$ is called a birationally fully faithful subcategory of $\mathcal{C}or_{k,\Lambda}^{prop}$, if the functor ${S_b^{-1}\mathcal{C} \to S_b^{-1}\mathcal{C}or_{k,\Lambda}^{prop}}$ is fully faithful.
\end{defi}

\begin{examp}
Clearly, $\mathcal{C}or_{k,\Lambda}^{prop}$ is a birationally fully faithful subcategory of itself. By \cite[Thm. 6.4]{KS1}, the subcategory $\mathcal{C}or_{k,\Lambda}^{pv} \subseteq \mathcal{C}or_{k,\Lambda}^{prop}$ is also birationally fully faithful.
\end{examp}

For a subcategory $\mathcal{C} \subseteq \mathcal{C}or_{k,\Lambda}^{prop}$, we denote $\mathbf{Im}_\Lambda^{\mathcal{C}}\mathbf{H}_0^S$ for the full subcategory of $\mathbf{NST}_k(\Lambda)$ spanned by sheaves with transfers isomorphic to $\mathbf{H}_0^S(X;\Lambda)$ for some $X \in \mathcal{C}$. We prove universal birational invariance of the zeroth Suslin homology.

\begin{thm}\label{univ. br. inv. tr.}
Assume $k$ perfect and admits a resolution of singularities. Let $\mathcal{C}$ be a birationally fully faithful subcategory of $\mathcal{C}or_{k,\Lambda}^{prop}$ and $\mathcal{A}$ be an arbitrary category enriched by $\Lambda$-modules.
\begin{enumerate}
\item Let $F : \mathcal{C} \to \mathcal{A}$ be an arbitrary $\Lambda$-linear functor which sends each birational morphism to an isomorphism. Then there exists one and only one (up to a natural equivalence) $\Lambda$-linear functor ${F_{S_b} : \mathbf{Im}_\Lambda^{\mathcal{C}}\mathbf{H}_0^S \to \mathcal{A}}$ such that the diagram
\begin{equation*}
\xymatrix{
\mathcal{C} \ar[r]^-F \ar[d] &\mathcal{A} \\
\mathbf{Im}_\Lambda^{\mathcal{C}}\mathbf{H}_0^S \ar@{.>}[ru]_-{F_{S_b}}
}
\end{equation*}
is $2$-commutative.
\item Let $F' : \mathcal{C}^{op} \to \mathcal{A}$ be an arbitrary $\Lambda$-linear functor which sends each birational morphism to an isomorphism. Then there exists one and only one (up to a natural equivalence) $\Lambda$-linear functor ${F'_{S_b} : \mathbf{Im}_\Lambda^{\mathcal{C}}\mathbf{H}_0^S \to \mathcal{A}}$ such that the diagram
\begin{equation*}
\xymatrix{
\mathcal{C}^{op} \ar[r]^-{F'} \ar[d] &\mathcal{A} \\
(\mathbf{Im}_\Lambda^{\mathcal{C}}\mathbf{H}_0^S)^{op} \ar@{.>}[ru]_-{F'_{S_b}}
}
\end{equation*}
is $2$-commutative.
\end{enumerate}
\end{thm}

\begin{proof}
By Proposition \ref{prop Suslin fully faithful}, the functors $\mathcal{C} \to \mathbf{Im}_\Lambda^{\mathcal{C}}\mathbf{H}_0^S$ and $\mathcal{C}^{op} \to (\mathbf{Im}_\Lambda^{\mathcal{C}}\mathbf{H}_0^S)^{op}$ are equivalences of categories. Thus this theorem follows from the universality of localizations of categories.
\end{proof}

\begin{rem}
Like Remark \ref{rmk H0A1 ext Modbr A}, the functor $F_{S_b}$ (resp. $F'_{S_b}$) in Theorem \ref{univ. br. inv. tr.} is canonically extended to ${\mathbf{NST}^{br}_k(\Lambda) \to \mathcal{A}}$ (resp. ${(\mathbf{NST}^{br}_k(\Lambda))^{op} \to \mathcal{A}}$) if ${\mathcal{C}or_{k,\Lambda}^{prop}} \subseteq \mathcal{C}$ and $\mathcal{A}$ is cocomplete.
\end{rem}

\subsection{Proper birational invariance}

In this subsection, we prove that a proper birational morphism of smooth (not necessary proper) $k$-varieties induces an isomorphism of the zeroth $\aone$- and Suslin homology. This is a refinement of Asok's result \cite[Thm. 3.9]{As}.

\begin{prop}\label{prop. br. inv.}
Let $f : X \to Y$ be a proper birational morphism of smooth $k$-varieties.
\begin{enumerate}
\item The induced morphisms
\begin{equation*}
\HAone{0}(X;\Lambda) \to \HAone{0}(Y;\Lambda)
\end{equation*}
is an isomorphisms of sheaves.
\item The induced morphisms
\begin{equation*}
\mathbf{H}_0^S(X;\Lambda) \to \mathbf{H}_0^S(Y;\Lambda)
\end{equation*}
is an isomorphism of sheaves with transfers.
\end{enumerate}
\end{prop}

\begin{proof}
By Lemma \ref{proper birationality}, for every ${M \in \Mod^{\aone}(\Lambda)}$ the induced map ${M(Y) \to M(X)}$ is an isomorphism. Thus \cite[Lem. 3.3]{As} gives isomorphisms
\begin{equation*}
\Hom_{\Mod^{\aone}(\Lambda)}(\HAone{0}(Y;\Lambda),M) \cong M(Y) \xrightarrow{\cong} M(X) \cong \Hom_{\Mod^{\aone}(\Lambda)}(\HAone{0}(X;\Lambda),M)
\end{equation*}
for all $M \in \Mod^{\aone}(\Lambda)$. Thus Yoneda's lemma in $\Mod^{\aone}(\Lambda)$ induces an isomorphism
\begin{equation*}
\HAone{0}(X;\Lambda) \xrightarrow{\cong} \HAone{0}(Y;\Lambda).
\end{equation*}
Similarly, \cite[Lem. 3.3]{As} also gives an isomorphism
\begin{align*}
\Hom_{\mathbf{HI}_k(\Lambda)}(\mathbf{H}_0^S(Y;\Lambda),M') \xrightarrow{\cong} \Hom_{\mathbf{HI}_k(\Lambda)}(\mathbf{H}_0^S(X;\Lambda),M')
\end{align*}
for all $M' \in \mathbf{HI}_k(\Lambda)$. Thus Yoneda's lemma in $\mathbf{HI}_k(\Lambda)$ induces an isomorphism
\begin{equation*}
\mathbf{H}_0^S(X;\Lambda) \xrightarrow{\cong} \mathbf{H}_0^S(Y;\Lambda)
\end{equation*}
in $\mathbf{HI}_k(\Lambda)$.
\end{proof}

\section{Applications to $\aone$-homotopy theory}\label{App}

In this section, we give some applications to $\aone$-homotopy theory. A morphism $X \to Y$ in $\Smk$ is called a $S^1$-stable $\aone$-$0$-equivalence, if the induced morphism
\begin{equation*}
\pi_0^{S}(\Sigma_{S^1}^{\infty}(X_+)) \to \pi_0^{S}(\Sigma_{S^1}^{\infty}(Y_+))
\end{equation*}
is an isomorphism. Note that $X \to Y$ is a $S^1$-stable $\aone$-$0$-equivalence if and only if the induced morphism $\HAone{0}(X;\mathbb{Z}) \to \HAone{0}(Y;\mathbb{Z})$ is an isomorphism by \cite[Prop. 2.1]{AH}. When $X$ and $Y$ are proper, we also obtain other equivalent conditions.

\begin{prop}\label{equiv. cond. of S1-stb 0-eq}
Assume $k$ admits a resolution of singularities. Let ${f : X \to Y}$ be a morphism of smooth proper $k$-schemes. Then the following conditions are equivalent.
\begin{enumerate}
\item\label{eq cond S1 stb 0 eq} The morphism $f$ is a $S^1$-stable $\aone$-$0$-equivalence.
\item\label{eq cond isom in Sb-1Smk} The morphism $f$ is an isomorphism in $S_b^{-1}\Smk$.
\item\label{eq cond bijec R-eq class} The induced map $X(K)/R \to Y(K)/R$ is bijective for all $K \in \mathcal{F}_k$.
\end{enumerate}
\end{prop}

\begin{proof}
Corollary \ref{cons. and faithful} shows that \eqref{eq cond S1 stb 0 eq} $\Rightarrow$ \eqref{eq cond isom in Sb-1Smk}. Moreover, \eqref{eq cond isom in Sb-1Smk} $\Rightarrow$ \eqref{eq cond bijec R-eq class} follows from \cite[Thm. 6.6.3]{KS2}. On the other hand, \eqref{eq cond bijec R-eq class} is equivalent to the condition that the induced morphism $\pi_0^{b\aone}(X) \to \pi_0^{b\aone}(Y)$ is an isomorphism by \cite[Thm. 6.2.1]{AM}. Thus Theorem \ref{str H0A1} shows that the induced map $\HAone{0}(X;\mathbb{Z}) \to \HAone{0}(Y;\mathbb{Z})$ is an isomorphism. Hence, we have \eqref{eq cond bijec R-eq class} $\Rightarrow$ \eqref{eq cond S1 stb 0 eq}.
\end{proof}

Next, we consider the $\aone$-connectedness of smooth proper varieties (see definition \cite[p.110]{MV}). In \cite{As}, Asok proves that a smooth proper $k$-variety $X$ is $\aone$-connected if and only if the structure morphism $X \to \Spec k$ induces an isomorphism of zeroth $\aone$-homology sheaves ${\HAone{0}(X;\Lambda) \xrightarrow{\cong} \Lambda}$.

\begin{rem}\label{rem A1-conn and isom to pt}
Proposition \ref{equiv. cond. of S1-stb 0-eq} implies that a smooth proper $k$-variety $X$ is $\aone$-connected if and only if the structure morphism $X \to \Spec k$ is an isomorphism in $S_b^{-1}\Smk$, when $k$ admits a resolution of singularities. However, this assertion holds without the assumption on resolution of singularities. Indeed, $X$ is $\aone$-connected if and only if $\#(X(K)/R)=1$ for all $K \in \mathcal{F}_k$ by \cite[Cor. 2.4.4]{AM}. On the other hand, by \cite[Thm. 8.5.1]{KS2} this is equivalent to the condition that the morphism $X \to \Spec k$ is an isomorphism in $S_b^{-1}\Smk$.
\end{rem}

The following gives another equivalent condition of the $\aone$-connectedness of $X$. 

\begin{prop}
Assume $k$ admits a resolution of singularities. A smooth proper $k$-variety $X$ is $\aone$-connected if and only if
\begin{equation*}
b_0^{\aone}(X_{k(X)}) \leq 1 \leq b_0^{\aone}(X).
\end{equation*}
\end{prop}

\begin{proof}
``Only if'' follows from \cite[Thm. 4.14]{As}. We prove ``if''. Since $b_0^{\aone}(X) \neq 0$, the variety $X$ has a $k$-rational point by Theorem \ref{A1-betti and rat pt}. Moreover, since $X_{k(X)}$ also has a $k(X)$-rational point, we have ${b_0^{\aone}(X_{k(X)}) = 1}$. This implies
\begin{equation*}
\#(X(k(X))/R) = \#(X_{k(X)}(k(X))/R) = 1
\end{equation*}
by Corollary \ref{free R-mod of R-eq}. Thus \cite[Thm. 8.5.1.]{KS2} shows that the structure morphism $X \to \Spec k$ is an isomorphism in $S_b^{-1}\Smk$. Therefore, $X$ is $\aone$-connected (see Remark \ref{rem A1-conn and isom to pt}). 
\end{proof}

Next, we prove a comparison result between $\aone$-homotopy and ordinary or \'etale homotopy. For a complex variety $X$, we denote $X^{an}$ for the associated analytic space.

\begin{prop}
Assume $k$ admits a resolution of singularities. Let ${X \to Y}$ be a $S^1$-stable $\aone$-$0$-equivalence of smooth proper $k$-varieties.
\begin{enumerate}
\item The induced map $\pi_1^{\acute{e}t}(X) \to \pi_1^{\acute{e}t}(Y)$ is an isomorphism.
\item If $k = \mathbb{C}$, the continuous map $X^{an} \to Y^{an}$ is a $1$-equivalence of topological spaces.
\end{enumerate}
\end{prop}

\begin{proof}
Since the functor of (\'etale) fundamental groups is birational invariant of smooth proper varieties, this is regarded as a functor on $S_b^{-1}\Smk^{prop}$. On the other hand, $X \to Y$ is an isomorphism in $S_b^{-1}\Smk^{prop}$ by Proposition \ref{equiv. cond. of S1-stb 0-eq}. Thus the map of (\'etale) fundamental groups induced by $X \to Y$ is an isomorphism.
\end{proof}

\begin{defi}
A smooth $k$-scheme $X$ is called \textit{$R$-rigid}, if for all $K \in \mathcal{F}_k$ the scalar extension $X_K$ has no rational curves over $K$.
\end{defi}

By \cite[Thm. 6.2.1]{AM}, a smooth proper variety $X$ is $R$-rigid if and only if the canonical morphism of sheaves $X \to \pi_0^{b\aone}(X)$ is an isomorphism. We see examples of $R$-rigid varieties.

\begin{examp}
\begin{enumerate}
\item Let $C$ be a geometrically irreducible smooth projective curve of genus $\geq 1$. Then $C$ is $R$-rigid. Indeed, for every $K \in \mathcal{F}_k$ the scalar extension $C_K$ has genus $\geq 1$ and thus $C_K$ has no rational curves over $K$.
\item An abelian variety $X$ is also $R$-rigid. Indeed, for every $K \in \mathcal{F}_k$ the scalar extension $X_K$ is a copoduct of abelian varieties over $K$ and thus $X_K$ has no rational curves over $K$.
\end{enumerate}
\end{examp}

We give a classification of $R$-rigid varieties up to a $S^1$-stable $\aone$-$0$-equivalences.

\begin{prop}
Assume $k$ admits a resolution of singularities. Let ${f : X \to Y}$ be a morphism of $R$-rigid smooth proper $k$-varieties. Then the following conditions are equivalent.
\begin{enumerate}
\item\label{R-rig. isom} $f$ is an isomorphism in $\Smk$.
\item\label{R-rig. birat} $f$ is birational.
\item\label{R-rig. stab} $f$ is stable birational.
\item\label{R-rig. w.eq} $f$ is an $\aone$-weak equivalence.
\item\label{R-rig. 0-eq} $f$ is a $S^1$-stable $\aone$-0-equivalence.
\end{enumerate}
\end{prop}

\begin{proof}
\eqref{R-rig. isom} $\Rightarrow$ \eqref{R-rig. birat} $\Rightarrow$ \eqref{R-rig. stab} and \eqref{R-rig. isom} $\Rightarrow$ \eqref{R-rig. w.eq} $\Rightarrow$ \eqref{R-rig. 0-eq} are obvious. We first prove that \eqref{R-rig. stab} $\Rightarrow$ \eqref{R-rig. 0-eq}. By \cite[Thm. 1.7.2]{KS2}, $X \to Y$ is an isomorphism in $S_b^{-1}\Smk$. Thus by Proposition \ref{equiv. cond. of S1-stb 0-eq}, $X \to Y$ is a $S^1$-stable $\aone$-0-equivalence. Next, we prove that \eqref{R-rig. 0-eq} $\Rightarrow$ \eqref{R-rig. isom}. Then the induced map $X(k(U))/R \to Y(k(U))/R$ is bijective for all $U \in \Smk^{var}$ by Proposition \ref{equiv. cond. of S1-stb 0-eq}. Hence, Lemma \ref{Req and pibr} shows that the morphism of birational sheaves
\begin{equation}\label{eq. in pf of R-rig}
\pi_0^{b\aone}(X) \to \pi_0^{b\aone}(Y)
\end{equation}
is an isomorphism. On the other hand, since $X$ and $Y$ are $R$-rigid, there exist canonical isomorphisms of sheaves $X \cong \pi_0^{b\aone}(X)$ and $Y \cong \pi_0^{b\aone}(Y)$. Thus by Yoneda's lemma in $\Smk$, \eqref{eq. in pf of R-rig} induces an isomorphism $X \xrightarrow{\cong} Y$. 
\end{proof}


\begin{thebibliography}{99}
  \bibitem[Aso12]{As} A. Asok, \textit{Birational invariants and $\mathbb{A}^1$-connectedness}, J. Reine Angew. Math., \textbf{681} (2012) 39-64.
  \bibitem[AH11]{AH} A. Asok and C. Haesemeyer, \textit{Stable $\mathbb{A}^1$-homotopy and $R$-equivalence}, J. Pure Appl. Algebra \textbf{215}(10) (2011) 2469–2472.
  \bibitem[AM11]{AM} A. Asok and F. Morel, \textit{Smooth varieties up to $\mathbb{A}^1$-homotopy and algebraic h-cobordisms}, Adv. Math. \textbf{227}(5) (2011) 1990-2058.
  \bibitem[CT95]{CT} J.-L. Colliot-Th\'el\`ene, \textit{Birational invariants, purity and the Gersten conjecture}, K-theory and algebraic geometry: connections with quadratic forms and division algebras (Santa Barbara, CA, 1992), Proc. Sympos. Pure Math., \textbf{58}, Part 1, Amer. Math. Soc., Providence, RI (1995) 1–64.
  \bibitem[Fal83]{Fa} G. Faltings, \textit{Endlichkeitss\"{a}tze f\"{u}r abelsche Variet\"{a}ten \"{u}ber Zahlk\"{o}rpern}, Invent. Math., \textbf{73} (1983) 349-366.
  \bibitem[GZ67]{GZ} P. Gabriel and M. Zisman, \textit{Calculus of fractions and homotopy theory}, Ergebnisse der Mathematik und ihrer Grenzgebiete, Band \textbf{35}, Springer (1967)
  \bibitem[KS07]{KS1} B. Kahn and R. Sujatha, \textit{A few localisation theorems}, Homology Homotopy Appl., \textbf{9}(2) (2007) 137–161.
  \bibitem[KS15]{KS2} B. Kahn and R. Sujatha, \textit{Birational geometry and localisation of categories}, Doc. Math., Extra vol.: Alexander S. Merkurjev's sixtieth birthday (2015) 277–334.
  \bibitem[KS17]{KS3} B. Kahn and R. Sujatha, \textit{Birational motives, I\hspace{-.1em}I: Triangulated birational motives}, Int. Math. Res. Not. IMRN, \textbf{22} (2017) 6778–6831.
  \bibitem[Kol99]{Ko} J. Koll\'ar, \textit{Rationally connected varieties over local fields}, Ann. of Math., \textbf{150} (1999) 357-367.
  \bibitem[Man86]{Ma} Yu.I. Manin, \textit{Cubic forms}, Algebra, geometry, arithmetic. Translated from the Russian by M. Hazewinkel. Second edition. North-Holland Mathematical Library, \textbf{4}. North-Holland Publishing Co., Amsterdam, 1986.
  \bibitem[Mor05]{Mo1} F. Morel, \textit{The stable $\mathbb{A}^1$-connectivity theorems}, K-theory, \textbf{35}(1-2) (2005) 1-68.
  \bibitem[Mor12]{Mo2} F. Morel, \textit{$\mathbb{A}^1$-algebraic topology over a field}, Lecture Notes in Math., \textbf{2052}, Springer, Heidelberg (2012).
  \bibitem[MV99]{MV} F. Morel and V. Voevodsky, \textit{$\mathbb{A}^1$-homotopy theory of schemes}, Publ. Math. IHES, \textbf{90} (1999) 45-143.
  \bibitem[MVW]{MVW} V. Mazza, V. Voevodsky, and C. Weibel, \textit{Lecture notes on motivic cohomology}, volume 2 of Clay Mathematics Monographs, America Mathematical Society, Providence, RI (2006).
  \bibitem[SV96]{SV} A. Suslin and V. Voevodsky, \textit{Singular homology of abstract varieties}, Invent. Math., \textbf{123}(1) (1996) 61-94.
  \bibitem[Voe00]{Vo00} V. Voevodsky, \textit{Triangulated categories of motives over a field}, Cycles, transfers, and motivic homology theories, Ann. of Math. Stud., \textbf{143}, Princeton Univ. Press, Princeton, NJ (2000) 188-238.
\end{thebibliography}
\end{document}